\documentclass[11pt]{article}
\usepackage{amsmath,amssymb,amsthm,amsfonts,amstext,amsbsy,amscd,mathrsfs,graphics}

\newtheorem{defin}{Definition}
\newtheorem{lemma}{Lemma}
\newtheorem{thm}{Theorem}
\newtheorem{remark}{Remark}

\newtheorem{prop}{Proposition}
\newtheorem{corl}{Corollary}

\usepackage{amsmath}
\usepackage{amssymb}
\usepackage{graphicx}

\usepackage{color}

\def\Z{\mathbb Z}

\newcommand{\N}{\mathbb N}
\newcommand{\R}{\mathbb R}
\newcommand{\Q}{\mathbb Q}
\begin{document}
\title{Tanaka's equation on the circle and stochastic flows}
\date{\today}
\maketitle
\begin{center}
\renewcommand{\thefootnote}{(\arabic{footnote})}
  \scshape Hatem Hajri\footnote{Universit\'e du Luxembourg, Email: Hatem.Hajri@uni.lu}
  and Olivier Raimond\footnote{Universit\'e Paris Ouest Nanterre La D\'efense, Email: oraimond@u-paris10.fr}
\renewcommand{\thefootnote}{\arabic{footnote}}\setcounter{footnote}{0}
\end{center}
\hglue0.02\linewidth\begin{minipage}{0.9\linewidth}
\end{minipage}

\maketitle

%\begin{document}

%\title[Tanaka's equation on the circle and stochastic flows]{Tanaka's equation on the circle and stochastic flows}

%\author{Hatem Hajri}
%\author{Olivier Raimond}

%\address{Universit\'e du Luxembourg\newline
%Campus Kirchberg, 6 rue Richard Coudenhove-Kalergi\newline
%L-1359 Luxembourg, Grand-Duchy of Luxembourg}
%\email{hatem.hajri@uni.lu}
%\address{Universit\'e Paris Ouest Nanterre La D\'efense\newline
%B\^atiment G, 200 avenue de la R\'epublique 92000 Nanterre, France.}
%\email{olivier.raimond@u-paris10.fr}
%\subjclass[2000]{60H25, 60J60.} 
%\keywords{Tanaka's SDE, stochastic flows of kernels, Brownian motion on the circle.}

\begin{abstract}
  We define a Tanaka's equation on an oriented graph with two edges and two vertices. This graph will be embedded in the unit circle. Extending this equation to flows of kernels, we show that the laws of the flows of kernels $K$ solutions of Tanaka's equation can be classified by pairs of probability measures $(m^+,m^-)$ on $[0,1]$, with mean $1/2$. What happens at the first vertex is governed by $m^+$, and at the second by $m^-$. For each vertex $P$, we construct a sequence of stopping times along which the image of the whole circle by $K$ is reduced to $P$. We also prove that the supports of these flows contain a finite number of points, and that except for some particular cases this number of points can be arbitrarily large.
\end{abstract}

\maketitle

\section{Introduction}

Consider Tanaka's equation
\begin{equation}\label{cc}
\varphi_{s,t}(x)=x+\int_{s}^t \textrm{sgn}(\varphi_{s,u}(x))dW_{u},\ \ s\leq t,\; x\in\R,
\end{equation}
where $(W_{t})_{t\in\R}$ is a Brownian motion on $\R$ (that is $(W_{t})_{t\geq 0}$ and $(W_{-t})_{t\geq 0}$ are two independent standard Brownian motions) and $\varphi=(\varphi_{s,t};\;s\le t)$ is a stochastic flow of mappings on $\R$. We refer to \cite{MR2060298} for a precise definition. Roughly, $\varphi_{s,t}$ and $\varphi_{0,t-s}$ are equal in law, for any sequence $\{[s_i,t_i], 1\le i\le n\}$ of non-overlapping intervals the mappings $\varphi_{s_i,t_i}$ are independent, and we have the flow property: for all $x\in\R$, $s\le t\le u$, a.s. $\varphi_{s,u}(x)=\varphi_{t,u}\circ \varphi_{s,t}(x)$. 
In \cite{MR2235172}, (\ref{cc}) is extended (\ref{cc}) to flows of kernels. A stochastic flow of kernels $K=(K_{s,t};\;s\le t)$ is the same as a stochastic flow of mappings, but the mappings are replaced by kernels, and the flow property being now that  for all $x\in\R$, $s\le t\le u$, a.s. $K_{s,u}(x)=K_{s,t}K_{t,u}(x)$  (with the usual composition of kernels). 
For $x\in\R$ and $s\le t$, $K_{s,t}(x)$ is a probability measure on $\R$ which describes the transport by the flow of a Dirac measure at $x$ from time $s$ to time $t$. A simple example of flow of kernels is $K_{s,t}(x)=\delta_{\varphi_{s,t}(x)}$, where $\varphi$ is a stochastic flow of mappings.  
%can be viewed as the transition probabilities of a Markov process in a random environment. 
%The first examples of diffusive flows of kernels, i.e. which are not of the form $\delta_{\varphi}$, are those solving Tanaka's equation extended to kernels. 

By applying It\^o's formula, it is easy to see that $(\varphi,W)$ solves (\ref{cc}) if and only if, setting $K=\delta_{\varphi}$, we have for all $s\leq t$, $x\in\R$ and $f\in C^2_b(\R)$ ($f$ is $C^2$ on $\R$ and $f',f''$ are bounded), a.s.
\begin{equation}\label{dd}
K_{s,t}f(x)=f(x)+\int_s^tK_{s,u}(f'\textrm{sgn})(x)dW_{u}+\frac{1}{2}\int_s^t K_{s,u}f''(x)du.
\end{equation}
Now, if $K$ is a stochastic flow of kernels and $W$ is a Brownian motion on $\R$, we will say that $(K,W$) solves Tanaka's equation if and only if \eqref{dd} holds for all  $s\leq t$, $x\in\R$ and $f\in C^2_b(\R)$. To give an intuitive meaning of this SDE, the transport by a solution $K$ is governed by $W$ on $]0,\infty[$ and by $-W$ on  $]-\infty,0[$, but with possible splitting at $0$. We will also be interested in diffusive solutions of Tanaka's equation, i.e. solutions $K$ that cannot be written in the form $\delta_{\varphi}$. 
%Now equation (\ref{dd}), where the unknown $K$ is a flow of kernels in general is called Tanaka's equation extended to kernels. 
The main result of \cite{MR2235172} is a one-to-one correspondence between probability measures $m$ on $[0,1]$ with mean $\frac{1}{2}$ and laws of solutions to (\ref{dd}). Denote by $\mathbb P^m$, the law of the solution $(K,W)$ associated to $m$. Then
$$K_{s,t}(x)=\delta_{x+\textrm{sgn}(x)W_{s,t}} 1_{\{t\leq \tau_{s,x}\}}+(U_{s,t}\delta_{W^+_{s,t}}+(1-U_{s,t})\delta_{-W^+_{s,t}}) 1_{\{t>\tau_{s,x}\}}$$
where $W_{s,t}=W_t-W_s,\ W_{s,t}^+=W_{t}-\displaystyle\inf_{u\in[s,t]}W_{u}=W_{s,t}-\displaystyle\inf_{u\in[s,t]}W_{s,u},$
$$\tau_{s,x}=\inf\{t\geq s: W_{s,t}=-|x|\}$$
and where $U_{s,t}$ is independent of $W$, with law $m$. In particular, when $m=\delta_{\frac{1}{2}}$, then $U_{s,t}=\frac{1}{2}$ and $K$ is $\sigma(W)$-measurable; this is also the unique $\sigma(W)$-measurable solution of (\ref{dd}).
For $m=\frac{1}{2}(\delta_0+\delta_1)$, we recover the unique flow of mappings solving (\ref{cc}) which was firstly introduced in \cite{MR1816931}. In \cite{MR2835247}, a more general Tanaka's equation has been defined on a graph related to Walsh's Brownian motion. In this work, we deal with another simple oriented graph with two edges and two vertices that will be embedded in the unit circle $\mathscr C=\{z\in\mathbb C : |z|=1\}$.\\
A function $f$ defined on $\mathscr C$ is said to be derivable in $z_0\in\mathscr C$ if
$$f'(z_0):=\lim_{h\rightarrow0}\frac{f(z_0 e^{ih})-f(z_0)}{h}$$
exists.
Let $C^2(\mathscr C)$ be the space of all functions $f$ defined on $\mathscr C$ having first and second continuous derivatives $f'$ and $f''$. Let $\mathcal{P}(\mathscr C)$ be the space of all probability measures on $\mathscr C$ and $(f_n)_{n\in\N}$ be a sequence of functions dense in $\{f\in C(\mathscr C), ||f||_{\infty}\leq 1\}$. We equip $\mathcal{P}(\mathscr C)$ with the following distance $d$ and its associated Borel $\sigma$-field:
\begin{equation}\label{sare}
d(\mu,\nu)=\left(\sum_{n}2^{-n}\left(\int f_nd\mu-\int f_nd\nu\right)^2\right)^{\frac{1}{2}} \ \textrm{with}\ \mu, \nu\in \mathcal{P}(\mathscr C).
\end{equation}
In the following, $\arg(z)\in[0,2\pi[$ denotes the argument of $z\in\mathbb C$ and in all the paper $l$ is a fixed parameter in $]0,\pi]$. Define for $z\in\mathscr C$, $$\epsilon(z)={1}_{\{arg(z)\in[0,l]\}}-1_{\{arg(z)\in]l,2\pi[\}}$$
and denote by $\mathscr C_l$ (or simply by $\mathscr C$ since $l$ will not vary) the graph embedded in $\mathscr C$ with two vertices $1$ and $e^{il}$ and two edges
$\mathscr C^+=\{z\in\mathscr C : arg(z)\in]0,l[\}$ and $\mathscr C^-=\mathscr C\setminus\mathscr C^+$ with orientation given by $\varepsilon$ (see Figure 1 below).

%%%%%%%%%%%%FIGURE%%%%%%%%%%%%%%%%%%
\begin{figure}[h]
\begin{center}
\resizebox{4cm}{4cm}{\input{fig2_article3.pstex_t}}
\caption{The graph $\mathscr C$.}
\end{center}
\end{figure}

\begin{defin}
On a probability space $(\Omega,\mathcal A,\mathbb P)$, let $W$ be a Brownian motion on $\R$ and $K$ be a stochastic flow of kernels on $\mathscr C$. We say that $(K,W)$ solves Tanaka's equation on $\mathscr C$ denoted $(T_{\mathscr C})$ if for all $s\leq t$, $f\in C^2(\mathscr C)$ and $x\in \mathscr C$, as.
\begin{equation}\label{jhk}
K_{s,t}f(x)=f(x)+\int_s^tK_{s,u}(\epsilon f')(x)dW_{u} + \frac{1}{2}\int_s^tK_{s,u}f''(x)du.
\end{equation}
If $(K,W)$ is a solution of $(T_{\mathscr C})$ and $K=\delta_\varphi$ with $\varphi$ a stochastic flow of mappings, we simply say that $(\varphi,W)$ solves $(T_{\mathscr C})$.
\end{defin}

\noindent If $(K,W)$ is a solution of $(T_{\mathscr C})$, then following Lemma 3.1 of \cite{MR2235172}, we have $\sigma(W)\subset\sigma(K)$ (see Lemma \ref{rz} (ii) below). So we will simply say that $K$ solves $(T_{\mathscr C})$.\\
In this paper, given two probability measures on $[0,1]$, $m^+$ and $m^-$ with mean $\frac{1}{2}$, we construct a flow $K^{m^+, m^-}$ solution of $(T_{\mathscr C})$. Let $(K^+,K^-,W)$ be such that given $W$, the flows $K^+$ and $K^-$ are independent and $(K^{\pm},\pm W)$ has for law $\mathbb P^{m^{\pm}}$. The flows $K^+$ and $K^-$ provide the additional randomness when $K^{m^+,m^-}$ passes through $1$ or $e^{il}$. Away from these two points, $K^{m^+,m^-}$ just follows $W$ on $\mathscr C^+$ and $-W$ on $\mathscr C^-$. We now state our first result.
\begin{thm}\label{a}  
\begin{itemize}
\item[(1)]  Let $m^+$ and $m^-$ be two probability measures on $[0,1]$ satisfying
\begin{equation}\label{mp}
\int_{0}^{1} u\ m^+ (du)=\int_{0}^{1} u\ m^- (du)=\frac{1}{2}.
\end{equation}
There exist a stochastic flow of kernels (unique in law) $K^{m^+, m^-}$ and a Brownian motion $W$ on $\R$ such that $(K^{m^+, m^-},W)$ solves $(T_{\mathscr C})$ and such that if $W_{s,t}^+=W_{t}-\displaystyle\inf_{u\in[s,t]}W_{u}$,  $W^{-}_{s,t}=\displaystyle\sup_{u\in[s,t]}W_{u}-W_t$ and
$$\rho_s=\inf\{t\geq s,\ \sup(W^{+}_{s,t},W^{-}_{s,t})=l\},$$
then conditionally to $\{s\leq t<\rho_s\}$, a.s.
\begin{eqnarray}
K^{m^+,m^-}_{s,t}(1)&=&U^+_{s,t}\delta_{\exp(iW^+_{s,t})}+(1-U^{+}_{s,t})\delta_{\exp(-iW^+_{s,t})},\nonumber\\
K^{m^+,m^-}_{s,t}(e^{il})&=&U^-_{s,t}\delta_{\exp(i(l+W^-_{s,t}))}+(1-U^{-}_{s,t})\delta_{\exp(i(l-W^-_{s,t}))}\nonumber\
\end{eqnarray}
and conditionally to $\{s\leq t<\rho_s\}$, $(U^+_{s,t},U^-_{s,t})$ is independent of $W$ and has for law $m^+\otimes m^-$ .\\
\item[(2)]  For all flow $K$ solution of $(T_{\mathscr C})$, there exists a unique pair of probability measures $(m^+,m^-)$ satisfying (\ref{mp}) such that $K\overset{law}{=}K^{m^+,m^-}$.
\end{itemize}
\end{thm}
\noindent Contrary to Tanaka's equation, where flows are concentrated on at most two points, flows associated to $(T_{\mathscr C})$ have nontrivial supports. The version $(K^{m^+, m^-},W)$ defined in Theorem \ref{a} (1), and constructed in Section 2, satisfies Proposition \ref{telv} and Proposition \ref{nahi} below. Proposition \ref{telv} shows the existence of some times at which the support of $K^{m^+, m^-}$ is only concentrated on a single point. For all $-\infty\leq s\leq t\leq +\infty$, let
\begin{equation}\label{jij}
\mathcal F^W_{s,t}=\sigma(W_{u,v},\ s\leq u\leq v\leq t).
\end{equation}
\begin{prop}\label{telv}
\begin{itemize}
\item[(1)] There exists an increasing sequence $(S_k)_{k\geq1}$ of $(\mathcal F^W_{0,t})_{t\geq0}$-stopping times such that a.s.  $\lim_{k\rightarrow\infty} S_k=+\infty$ and $K^{m^+,m^-}_{0,S_k}(z)=\delta_{e^{il}}$ for all $z\in\mathscr C$ and all $k\geq 1$.
\item[(2)] There exists an increasing sequence $(T_k)_{k\geq1}$ of $(\mathcal F^W_{0,t})_{t\geq0}$-stopping times such that a.s. $\lim_{k\rightarrow\infty} T_k=+\infty$ and $K^{m^+,m^-}_{0,T_k}(z)=\delta_{1}$ for all $z\in\mathscr C$ and all $k\geq 1$.
\end{itemize}
\end{prop}
The next proposition shows that the support of $K^{m^+,m^-}$ may contain an arbitrary large number of points with positive probability (more informations can be found in Section \ref{yiu}).
\begin{prop}\label{nahi}
Assume that $m^+$ and $m^-$ are both distinct from $\frac{1}{2}(\delta_{0}+\delta_1)$. Then there exists a sequence of events $(C_n)_{n\geq 0}$ and a sequence of $(\mathcal F^W_{0,t})_{t\geq0}$-stopping times $(\sigma_n)_{n\geq 0}$ such that for all $n\geq 0$,
\begin{enumerate}
 \item [(i)] $\mathbb P(C_n)>0$,
\item[(ii)] $\textrm{Card supp}\left(K^{m^+,m^-}_{0,\sigma_n}(1)\right)=n+1$ a.s. on $C_n$.
\end{enumerate}
\end{prop}
We also mention that all the sequences of stopping times discussed in the previous two propositions will be constructed independently of $(m^+,m^-)$. They take values in $\{\rho_n, n\in\N\}$ where $\rho_0=0$ and $\rho_{n+1}=\inf\{t\geq \rho_n,\ \sup(W^{+}_{\rho_n,t},W^{-}_{\rho_n,t})=l\}$ for $n\geq 0$. Set, for $z\in\mathscr C, n\in\N$,
$$X^z_n=\textrm{supp}\left(K^{m^+,m^-}_{0,\rho_n}(z)\right)$$
where $m^+$ and $m^-$ are distinct from $\frac{1}{2}(\delta_{0}+\delta_1)$. Then $(X^z_n)_n$ is a strong Markov chain on $E=\cup_{k\geq 1}\mathscr C^k.$ Proposition
\ref{telv} asserts that $\{1\}$ and $\{e^{il}\}$ are recurrent for this chain. Proposition \ref{nahi} asserts that for all $n\geq 0$, both $\{1\}$ and $\{e^{il}\}$ (by analogy) communicate with $\mathscr C^{n+1}$. So one can deduce the following immediate 
\begin{corl}\label{tila}
For all $z\in\mathscr C, n\geq 0$, $\mathscr C^{n+1}$ is a recurrent set for $X^z$ (i.e. a.s. $\forall n\geq 0,\ X^z_k \in \mathscr C^{n+1}$ for infinitely many $k$).
\end{corl}
Even that the supports of the flows $K^{m^+,m^-}$ may be concentrated on arbitrarily many points at some times, these random sets are always finite in the following sense: a.s. $$\forall z\in\mathscr C,\; t\geq 0,\quad \textrm{Card supp}\left(K^{m^+,m^-}_{0,t}(z)\right)<\infty.$$
Let us describe the organization of this paper. In Section 2, we prove the first part of Theorem \ref{a}. The proof of the second part will be the subject of Section 3.  In Section 4, we prove Proposition \ref{telv}. Section 5 gives some informations about the support of $K^{m^+,m^-}$ and proves Proposition \ref{nahi}.
\section{Construction of flows associated to $(T_{\mathscr C})$}\label{TAB}
Fix two probability measures $m^+$ and $m^-$ on $[0,1]$ with mean $\frac{1}{2}$.
\subsection{Coupling flows associated with two Tanaka's equations on $\R$. \label{klp} \newline}
In this section, we follow \cite{MR2235172}. By Kolmogorov extension theorem, there exists a probability space $(\Omega,\mathcal A,\mathbb P)$ on which one can construct a process $(\varepsilon_{s,t}^+,\varepsilon_{s,t}^-, U_{s,t}^+,U_{s,t}^-,W_{s,t})_{-\infty<s\leq t<\infty}$ taking values in $\{-1,1\}^2\times [0,1]^2\times \R$ such that (i), (ii), (iii), (iv) and (v) are satisfied, where
\begin{itemize}
 \item [(i)] $W_{s,t}:=W_t-W_s$ for all $s\leq t$ and $W$ is a Brownian motion on $\R$.
 \item[(ii)] Given $W$, $(\varepsilon_{s,t}^+,U_{s,t}^+)_{s\leq t}$ and $(\varepsilon_{s,t}^-,U_{s,t}^-)_{s\leq t}$ are independent.
\item[(iii)]  For fixed $s<t$, $(\varepsilon_{s,t}^{\pm},U_{s,t}^{\pm})$ is independent of $W$ and $$(\varepsilon_{s,t}^{\pm},U_{s,t}^{\pm})\overset{law}{=}(u\delta_{1}(dx)+(1-u)\delta_{-1}(dx))m^{\pm}(du).$$
\noindent In particular $\mathbb P(\varepsilon^{\pm}_{s,t}=1|U^{\pm}_{s,t})=U^{\pm}_{s,t}.$
\item[(iv)] Define for all $s\leq t$
$$\textrm{m}^+_{s,t}=\inf\{W_{u}; u\in[s,t]\}\quad \hbox{and} \quad \textrm{m}^-_{s,t}=\sup\{W_{u}; u\in[s,t]\}.$$
Then for all $s<t$ and $u<v$, then 
\begin{equation}\label{houd}
\mathbb P(\varepsilon^{\pm}_{s,t}=\varepsilon^{\pm}_{u,v}, U^{\pm}_{s,t}=U^{\pm}_{u,v}|\textrm{m}^{\pm}_{s,t}=\textrm{m}^{\pm}_{u,v})=1.
\end{equation}
\item[(v)] For all $s<t$ and $\{(s_{i},t_{i}); 1\leq i\leq n \}$ with $s_{i}<t_{i}$, the law of $(\varepsilon^{\pm}_{s,t},U^{\pm}_{s,t})$ knowing
  $(\varepsilon^{\pm}_{s_{i},t_{i}},U^{\pm}_{s_{i},t_{i}})_{1\leq i\leq n}$ and $W$ is given by
$$(u\delta_{1}(dx)+(1-u)\delta_{-1}(dx))m^{\pm}(du)$$
 when $\textrm{m}^{\pm}_{s,t}\not\in\{\textrm{m}^{\pm}_{s_{i},t_{i}}; 1\leq i\leq n\}$ and is otherwise given by
 $$\delta_{\varepsilon^{\pm}_{s_{i},t_{i}},U^{\pm}_{s_{i},t_{i}}}$$
 on the event $\{\textrm{m}^{\pm}_{s,t}=\textrm{m}^{\pm}_{s_{i},t_{i}}\}$ with $1\leq i\leq n$.
\end{itemize}

\medskip

\noindent Note that (i)-(v) uniquely define the law of $$(\varepsilon^+_{s_{1},t_{1}},U^+_{s_{1},t_{1}},\varepsilon^-_{s_{1},t_{1}},U^-_{s_{1},t_{1}},\cdots,\varepsilon^+_{s_{n},t_{n}},U^+_{s_{n},t_{n}},\varepsilon^-_{s_{n},t_{n}},U^-_{s_{n},t_{n}},W)$$ for all $s_i<t_i$, $1\leq i\leq n.$ This family of laws is consistent by construction.
Note in particular that, when (iv) is satisfied for $(s_i,t_i)$ and $(s_j,t_j)$ with $1\le i,j\le n$, then (v) properly defines the law of $(\varepsilon^{\pm}_{s,t},U^{\pm}_{s,t})$ knowing  $(\varepsilon^{\pm}_{s_{i},t_{i}},U^{\pm}_{s_{i},t_{i}})_{1\leq i\leq n}$ and $W$, and we have that (iv) also holds for $(s,t)$ and $(s_j,t_j)$ with $1\le j\le n$.\\
\noindent For $s\leq t$, $x\in \R$, define $$\tau^{\pm}_{s}(x)=\inf\{r\geq s:\; W_{s,r}=\mp|x|\}$$ and set
\begin{eqnarray}
\varphi^{\pm}_{s,t}(x)&=&(x\pm\textrm{sgn}(x)W_{s,t})1_{ \{t\leq \tau^{\pm}_{s}(x)\}}+\varepsilon^{\pm}_{s,t}W_{s,t}^{\pm}1_{ \{t> \tau^{\pm}_{s}(x)\}},\nonumber\\
K^{\pm}_{s,t}(x)&=&\delta_{x\pm\textrm{sgn}(x)W_{s,t}}1_{ \{t\leq \tau^{\pm}_{s}(x)\}}+\big(U^{\pm}_{s,t}\delta_{W_{s,t}^{\pm}}+(1-U^{\pm}_{s,t})\delta_{-W_{s,t}^{\pm}}\big) 1_{ \{t> \tau^{\pm}_{s}(x)\}}.\nonumber\
\end{eqnarray}
Recall the following
\begin{thm}(\cite{MR2235172}\label{mama})\\
(i) $(\varphi^{+},W)$ and $(\varphi^{-},-W)$  solve Tanaka's equation (\ref{cc}).\\
(ii) $(K^{+},W)$ and $(K^{-},-W)$  solve Tanaka's equation (\ref{dd}).\\
(iii) For all $x\in \R$, all $s\leq t$ and all bounded continuous function $f$, a.s.
$$K^{\pm}_{s,t}f(x)=E[f(\varphi^{\pm}_{s,t}(x))|K^{\pm}].$$
\end{thm}
\subsection{Modification of flows.\label{chayynt}\newline}
For our later needs, we will construct modifications of $\varphi^{\pm}$ and of $K^{\pm}$ which are measurable with respect to $(s,t,x,\omega)$.
On a set of probability $1$, define for all $s<t$, $(s_n,t_n)=(\frac{\lfloor ns \rfloor+1}{n},\frac{\lfloor nt \rfloor-1}{n})$ and
$$(\widetilde{\varepsilon}^{\pm}_{s,t},\widetilde{U}^{\pm}_{s,t})=(\limsup_{n\rightarrow\infty}{\varepsilon}^{\pm}_{s_n,t_n},\limsup_{n\rightarrow\infty}{U}^{\pm}_{s_n,t_n}).$$
\noindent Then, we have the following
\begin{lemma}\label{aa}
(i) For all $s<t$, a.s. $\widetilde{\varepsilon}^{\pm}_{s,t}={\varepsilon}^{\pm}_{s,t},\ \widetilde{U}^{\pm}_{s,t}={U}^{\pm}_{s,t}.$\\
(ii) Consider the random sets
$$\mathscr {D}^+=\{(s,t)\in\R^2; s<t, \textrm{m}^+_{s,t}<\textrm{min}(W_s,W_t)\},$$
$$\mathscr {D}^-=\{(s,t)\in\R^2; s<t, \textrm{m}^-_{s,t}>\textrm{max}(W_s,W_t)\}.$$
Then a.s. for all $(s,t)$ and $(u,v)$ in $\mathscr D^{\pm}$,
$$\textrm{m}^{\pm}_{s,t}=\textrm{m}^{\pm}_{u,v}\Longrightarrow(\widetilde{\varepsilon}^{\pm}_{s,t},\widetilde{U}^{\pm}_{s,t})=(\widetilde{\varepsilon}^{\pm}_{u,v},\widetilde{U}^{\pm}_{u,v}).$$
\end{lemma}
\begin{proof}
\textbf{(i)} By (\ref{houd}), a.s. for all $s<t, u<v$ such that $(s,t,u,v)\in\Q^4$, we have
$$m^{\pm}_{s,t}=m^{\pm}_{u,v}\Longrightarrow (\varepsilon^{\pm}_{s,t}, U^{\pm}_{s,t})=(\varepsilon^{\pm}_{u,v},U^{\pm}_{u,v}).$$
Fix $s<t$. With probability $1$, $\textrm{m}^{\pm}_{s,t}$ is attained in $]s,t[$ and thus a.s. there exists $n_0$ such that
\begin{equation}\label{hmom}
m^{\pm}_{s,t}=m^{\pm}_{s_n,t_n}=m^{\pm}_{s_{n_0},t_{n_0}}\ \textrm{for all}\ n\geq n_0.
\end{equation}
Taking the limit, we get $(\widetilde{\varepsilon}^{\pm}_{s,t},\widetilde{U}^{\pm}_{s,t})=({\varepsilon}^{\pm}_{s_{n_0},t_{n_0}},{U}^{\pm}_{s_{n_0},t_{n_0}})$ a.s. From (\ref{houd}) and (\ref{hmom}), we also have that $({\varepsilon}^{\pm}_{s,t},{U}^{\pm}_{s,t})=({\varepsilon}^{\pm}_{s_{n_0},t_{n_0}},{U}^{\pm}_{s_{n_0},t_{n_0}})$ a.s. and (i) is proved.
%$$(\varepsilon^{\pm}_{s_n,t_n}, U^{\pm}_{s_n,t_n})=(\varepsilon^{\pm}_{s,t},U^{\pm}_{s,t})\ \textrm{for all}\ n\geq n_0.$$

\textbf{(ii)} With probability $1$, for all $(s,t)$ and $(u,v)$ in $\mathscr D^{\pm}$, if $\textrm{m}^{\pm}_{s,t}=\textrm{m}^{\pm}_{u,v}$, then
$\exists n_0: \textrm{m}^{\pm}_{s_n,t_n}=\textrm{m}^{\pm}_{u_n,v_n}\ \textrm{for all}\ n\geq n_0$, which implies that $$\exists n_0: ({\varepsilon}^{\pm}_{s_{n},t_{n}},{U}^{\pm}_{s_{n},t_{n}})=({\varepsilon}^{\pm}_{u_{n},v_{n}},{U}^{\pm}_{u_{n},v_{n}})\ \textrm{for all}\ n\geq n_0$$
and thus that $\widetilde{\varepsilon}^{\pm}_{s,t}=\widetilde{\varepsilon}^{\pm}_{u,v}$ and that $\widetilde{U}^{\pm}_{s,t}=\widetilde{U}^{\pm}_{u,v}$.
\end{proof}

\noindent We may now consider the following modifications of ${\varphi}^{\pm}$ and ${K}^{\pm}$ defined for all $s\leq t, x\in \R$ by
\begin{eqnarray}
\widetilde{\varphi}^{\pm}_{s,t}(x)&=&(x\pm\textrm{sgn}(x)W_{s,t})1_{ \{t\leq \tau^{\pm}_{s}(x)\}}+\widetilde{\varepsilon}^{\pm}_{s,t}W_{s,t}^{\pm}1_{ \{t> \tau^{\pm}_{s}(x)\}},\nonumber\\
\widetilde{K}^{\pm}_{s,t}(x)&=&\delta_{x\pm\textrm{sgn}(x)W_{s,t}}1_{ \{t\leq \tau^{\pm}_{s}(x)\}}+\big(\widetilde{U}^{\pm}_{s,t}\delta_{W_{s,t}^{\pm}}+(1-\widetilde{U}^{\pm}_{s,t})\delta_{-W_{s,t}^{\pm}}\big) 1_{ \{t> \tau^{\pm}_{s}(x)\}}.\nonumber\
\end{eqnarray}
Then Theorem \ref{mama} holds also for $\widetilde{\varphi}^{\pm}, \widetilde{K}^{\pm}$ (because (i), (ii), (iii) and (iv) stated at the begining of Section \ref{klp} are satisfied by $(\widetilde{\varepsilon}^{\pm},\widetilde{U}^{\pm},W)$).
\begin{lemma}\label{feda}

(i) The mapping $$(s,t,x,\omega)\longmapsto (\widetilde{\varphi}^{\pm}_{s,t}(x,\omega),\widetilde{K}^{\pm}_{s,t}(x,\omega))$$
is measurable from $\{(s,t,x,\omega), s\leq t, x\in\R, \omega\in\Omega\}$ into $\R\times\mathcal P(\R)$.\\
(ii) For all $s,t,x$, a.s.
$${\varphi}^{\pm}_{s,t}(x)=\widetilde{\varphi}^{\pm}_{s,t}(x)\ \textrm{and}\ \ {K}^{\pm}_{s,t}(x)=\widetilde{K}^{\pm}_{s,t}(x).$$
\end{lemma}
\begin{proof}
\textbf{(i)} Clearly, the mapping
 $$(s,t,\omega)\longmapsto (\widetilde{\varepsilon}^{\pm}_{s,t}(\omega), \widetilde{U}^{\pm}_{s,t}(\omega), W_{s,t}(\omega))$$
 is measurable. For all $t\geq s$, we have
$$\{\tau^{+}_s(x)>t\}=\{\inf_{s\leq r\leq t}W_{s,r}+|x|>0\}$$
which shows that $(s,x,\omega)\longmapsto \tau^{+}_s(x,\omega)$ is measurable and a fortiori $(s,x,\omega)\longmapsto \tau^{-}_s(x,\omega)$ is also measurable. 
\textbf{(ii)} is a consequence of Lemma \ref{aa} (i).
\end{proof}
\noindent To simplify the notation, throughout the rest of the paper, we will denote $\widetilde{\varepsilon}^{\pm}_{s,t}, \widetilde{U}^{\pm}_{s,t},\widetilde{\varphi}^{\pm}_{s,t}, \widetilde{K}^{\pm}_{s,t}$ simply by $\varepsilon^{\pm}_{s,t}, U^{\pm}_{s,t},{\varphi}^{\pm}_{s,t}, K^{\pm}_{s,t}$.
\subsection{The construction of $K^{m^+,m^-}$. \label{ss}\newline}
In this paragraph, we construct a stochastic flow of kernels $K^{m^+,m^-}$ and a stochastic flow of mappings $\varphi$ respectively from $(K^+,K^-)$ and from $(\varphi^+,\varphi^-)$. Let
\begin{equation}\label{mom}
\rho_s=\inf\{r\geq s,\ \sup(W^{+}_{s,r},W^{-}_{s,r})=l\}.
\end{equation}
We first define $(\varphi_{s,t})_{s\leq t\leq \rho_s}$. For $t\in[s,\rho_s]$, set
\begin{eqnarray}
\varphi_{s,t}(1)&=&\exp(i\varphi^+_{s,t}(0)),\nonumber\\
\varphi_{s,t}(e^{il})&=&\exp(i(l+\varphi^-_{s,t}(0)))\nonumber\
\end{eqnarray}
and for $z\in\mathscr C\setminus\{1,e^{il}\}$ and $t\in[s,\rho_s]$, set
\begin{eqnarray}
\varphi_{s,t}(z)&=&ze^{i\epsilon(z)W_{s,t}} 1_{\{t\leq \tau_s(z)\}}\nonumber\\
&+&\left(\varphi_{s,t}(1)1_{\{ze^{i\epsilon(z)W_{s,\tau_{s}(z)}}=1\}}+\varphi_{s,t}(e^{il})1_{\{ze^{i\epsilon(z)W_{s,\tau_{s}(z)}}=e^{il}\}}\right)1_{\{t>\tau_s(z)\}},\nonumber\
\end{eqnarray}
where $$\tau_s(z)=\inf\{r\geq s,\ ze^{i\epsilon(z)W_{s,r}}=1\ \textrm{or}\ e^{il}\}.$$
Note that on $\{\tau_s(z)<\rho_s\}\cap\{ze^{i\epsilon(z)W_{s,\tau_{s}(z)}}=1\}$, we have $W^{+}_{s,\tau_s(z)}=0$ and consequently $\varphi_{s,\tau_s(z)}(1)=1$. Also, on $\{\tau_s(z)<\rho_s\}\cap\{ze^{i\epsilon(z)W_{s,\tau_{s}(z)}}=e^{il}\}$, we have $W^{-}_{s,\tau_s(z)}=0$ and so $\varphi_{s,\tau_s(z)}(e^{il})=e^{il}$.\\
Since $(s,\omega)\longmapsto \rho_s(\omega)$ and $(s,z,\omega)\longmapsto \tau_s(z,\omega)$ are measurable, it follows from Lemma \ref{feda} that $$(s,t,z,\omega)\longmapsto \varphi_{s,t}(z,\omega) 1_{\{s\leq t\leq \rho_s(\omega)\}}$$
is measurable from $\{(s,t,z,\omega), s\leq t, z\in\mathscr C,\omega\in\Omega\}$ into $\mathscr C$. Now we consider the sequence of stopping times $(\rho^k_s)_{k\geq 0}$ such that $\rho^0_s=s$ and $\rho^{k+1}_s=\rho_{\rho^k_s}$ for $k\geq 0$.\\
Define for all $s\leq t$,
$$\varphi_{s,t}=\sum_{k\geq0} 1_{\{\rho^k_s\leq t<\rho^{k+1}_s\}}\varphi_{\rho^k_s,t}\circ \varphi_{\rho^{k-1}_s,\rho^k_s}\circ\cdots\circ \varphi_{s,\rho_s}.$$
Then $(s,t,z,\omega)\longmapsto \varphi_{s,t}(z,\omega)$ is measurable from $\{(s,t,z,\omega), s\leq t, z\in\mathscr C,\omega\in\Omega\}$ into $\mathscr C$. By the same way, we define $(K^{m^+,m^-}_{s,t})_{s\leq t\leq\rho_s}$ for $t\in[s,\rho_s]$
\begin{eqnarray}
K^{m^+,m^-}_{s,t}(1)&=&U^+_{s,t}\delta_{\exp(iW^+_{s,t})}+(1-U^{+}_{s,t})\delta_{\exp(-iW^+_{s,t})},\nonumber\\
K^{m^+,m^-}_{s,t}(e^{il})&=&U^-_{s,t}\delta_{\exp(i(l+W^-_{s,t}))}+(1-U^{-}_{s,t})\delta_{\exp(i(l-W^-_{s,t}))}\nonumber\
\end{eqnarray}
and for $z\in\mathscr C\setminus\{1,e^{il}\}$ and $t\in[s,\rho_s]$
\begin{eqnarray}
K^{m^+,m^-}_{s,t}(z)&=&\delta_{ze^{i\epsilon(z)W_{s,t}}} 1_{\{t\leq \tau_s(z)\}}\nonumber\\
&+&\left(K^{m^+,m^-}_{s,t}(1)1_{\{ze^{i\epsilon(z)W_{s,\tau_{s}(z)}}=1\}}+K^{m^+,m^-}_{s,t}(e^{il})1_{\{ze^{i\epsilon(z)W_{s,\tau_{s}(z)}}=e^{il}\}}\right)1_{\{t>\tau_s(z)\}}.\nonumber\
\end{eqnarray}
Define now for all $s\leq t$,
$$K^{m^+,m^-}_{s,t}=\sum_{k\geq0} 1_{\{\rho^k_s\leq t<\rho^{k+1}_s\}}K^{m^+,m^-}_{s,\rho_s}\cdots K^{m^+,m^-}_{\rho^{k-1}_s,\rho^k_s}K^{m^+,m^-}_{\rho^k_s,t}.$$
Then $(s,t,z,\omega)\longmapsto K^{m^+,m^-}_{s,t}(z,\omega)$ is measurable from $\{(s,t,z,\omega), s\leq t, z\in\mathscr C,\omega\in\Omega\}$ into $\mathcal P(\mathscr C)$.\\
For every choice $s_1<t_1<\cdots<s_n<t_n$, $(\varphi_{s_i,t_i},K^{m^+,m^-}_{s_i,t_i})$ is $\sigma(\varepsilon_{u,v}^{+},\varepsilon_{u,v}^{-},U_{u,v}^{+},U_{u,v}^{-},W_{u,v}, s_i\leq u\leq v\leq t_i )$ measurable and these  $\sigma$-fields are independent for $1\leq i\leq n$ by construction. This implies the independence of the family $\{(\varphi_{s_i,t_i},K^{m^+,m^-}_{s_i,t_i}),\ 1\leq i\leq n\}$. It is also clear that the laws of $\varphi_{s,t}$ and $K^{m^+,m^-}_{s,t}$ only depend on $t-s$.

\subsection{The flow property for $K^{m^+,m^-}$ and $\varphi$.\newline}
To prove the flow property for both $\varphi$ and $K^{m^+,m^-}$, we start by the following
\begin{prop}\label{b}
Let $S$ and $T$ be two finite $(\mathcal F^W_{-\infty,r})_{r\in\R}$-stopping times such that $S \leq T\leq \rho_S$. Then a.s.  for all $u\in[T,\rho_S], z\in\mathscr C$, we have
$$\varphi_{S,u}(z)=\varphi_{T,u}\circ \varphi_{S,T}(z)$$
and $$K^{m^+,m^-}_{S,u}(z)=K^{m^+,m^-}_{S,T}K^{m^+,m^-}_{T,u} (z).$$
\end{prop}
\begin{proof}
Define
\begin{eqnarray}
\Omega_1&=&\{\omega\in\Omega : \; \forall\ (s_1,t_1), (s_2,t_2)\in \mathscr {D}^{\pm},\; m^{\pm}_{s_1,t_1}=m^{\pm}_{s_2,t_2}\Rightarrow \varepsilon^{\pm}_{s_1,t_1}=\varepsilon^{\pm}_{s_2,t_2}\}\nonumber\\
\Omega_2&=&\{\omega\in\Omega : \; m^+_{T,T+r}<W_T<m^-_{T,T+r},\; m^+_{S,S+r}<W_S<m^-_{S,S+r}\ \textrm{for all}\ r>0\}.\nonumber\
\end{eqnarray}
Then $\mathbb P(\Omega_1)=1$ (see Lemma \ref{aa} (ii)). It is also known that $\mathbb P(\Omega_2)=1$ (see \cite{MR1725357} page 94). We will prove the proposition on the set of probability $1$: $\tilde\Omega=\Omega_1\cap\Omega_2$ and we first prove the result for $\varphi$. From now on, we fix $\omega\in\tilde\Omega$. Define
\begin{eqnarray}
E_{(i)}&=&\{(u,z) :\;T\leq u\leq \rho_S, u<\tau_S(z)\},\nonumber\\
E_{(ii)}&=&\{(u,z) : \;T<\tau_S(z)\leq u\leq\rho_S\},\nonumber\\
E_{(iii)}&=&\{(u,z) :\; \tau_S(z)\leq T\leq u\leq \rho_S, u<\tau_T(\varphi_{S,T}(z))\},\nonumber\\
E_{(iv)}&=&\{(u,z) : \;\tau_S(z)\leq T\leq\tau_T(\varphi_{S,T}(z))\leq u\leq \rho_S\}.\nonumber\
\end{eqnarray}
Then $E_{(i)}\cup E_{(ii)}\cup E_{(iii)}\cup E_{(iv)}=[T,\rho_S]\times \mathscr C$. For all $z\in\mathscr C$, set $Z=\varphi_{S,T}(z)$ and $\theta=\arg(z)$.\\
\textbf{(i)} Let $(z,u)\in E_{(i)}$. Then as $T<\tau_S(z)$, we have $\theta\notin \{0,l\}, Z=ze^{i\epsilon(z)W_{S,T}}$ and
\begin{eqnarray}
\tau_T(Z)&=&\inf\{r\geq T,\ Ze^{i\epsilon(Z)W_{T,r}}=1\ \textrm{or}\ e^{il}\}\nonumber\\
&=&\inf\{r\geq T,\ ze^{i(\epsilon(z)W_{S,T}+\epsilon(Z)W_{T,r})}=1\ \textrm{or}\ e^{il}\}=\tau_S(z)\nonumber\
\end{eqnarray}
since $\epsilon(z)=\epsilon(Z)$. Therefore $u<\tau_T(Z)$ and $\varphi_{T,u}\circ \varphi_{S,T}(z)=Ze^{i\epsilon(Z)W_{T,u}}=ze^{i\epsilon(z)W_{S,u}}=\varphi_{S,u}(z)$.\\
\textbf{(ii)} Let $(z,u)\in E_{(ii)}$. Then, we still have $\tau_T(Z)=\tau_S(z)$ and $\varphi_{T,\tau_{T}(Z)}(Z)=\varphi_{S,\tau_{S}(z)}(z)$. Recall that
$$\varphi_{S,u}(z)=\varphi_{S,u}(1) 1_{\{\varphi_{S,\tau_S(z)}(z)=1\}}+\varphi_{S,u}(e^{il})1_{\{\varphi_{S,\tau_S(z)}(z)=e^{il}\}}$$
and $$\varphi_{T,u}(Z)=\varphi_{T,u}(1) 1_{\{\varphi_{T,\tau_T(Z)}(Z)=1\}}+\varphi_{T,u}(e^{il})1_{\{\varphi_{T,\tau_T(Z)}(Z)=e^{il}\}}.$$
Suppose for example $\varphi_{S,\tau_S(z)}(z)=\varphi_{T,\tau_T(Z)}(Z)=1 $, then $W^+_{T,\tau_T(Z)}=W^+_{S,\tau_S(z)}=0$ and so $W^+_{T,r}=W^+_{S,r}$ (and a fortiori $m^+_{T,r}=m^+_{S,r}$) for all $r\geq \tau_T(Z)(=\tau_S(z))$. From the definition,
$$\varphi_{S,u}(z)=\varphi_{S,u}(1)=\exp(i\varphi^+_{S,u}(0))\ \textrm{and}\ \ \varphi_{T,u}(Z)=\varphi_{T,u}(1)=\exp(i\varphi^+_{T,u}(0)).$$
If $W^+_{T,u}=W^{+}_{S,u}=0$, then $\varphi_{S,u}(z)=\varphi_{T,u}(Z)=1$. Suppose that $W^{+}_{T,u}=W^{+}_{S,u}>0$, then $W_u>\textrm{m}^+_{T,u}$ and $W_u>\textrm{m}^+_{S,u}$. Since $\omega\in\Omega_2$, we have $$W_T>\textrm{m}^+_{T,u}\ \textrm{and}\ W_S>\textrm{m}^+_{S,u}.$$
In other words, $(T,u)$ and $(S,u)$ are in $\mathscr D^+$ so that $\varepsilon^+_{S,u}=\varepsilon^+_{T,u}$ and $\varphi_{T,u}(Z)=\varphi_{S,u}(z)$.\\
\textbf{(iii)} Let $(z,u)\in E_{(iii)}$. Assume for example that $\varphi_{S,\tau_S(z)}(z)=1$, then $Z=\varphi_{S,T}(1)=e^{i\varphi^+_{S,T}(0)}$ since $T\leq\rho_S$ and
\begin{eqnarray}
\varphi_{T,u}(Z)&=&\exp(i(\varphi^+_{S,T}(0)+\epsilon(Z)W_{T,u}))\nonumber\\
&=&\exp(i(\varepsilon^+_{S,T}W^+_{S,T}+\epsilon(Z)W_{T,u})).\nonumber\
\end{eqnarray}
As $T\leq u<\tau_T(Z)$, it follows that $Z\notin \{1,e^{il}\}$ (if $Z\in \{1,e^{il}\}$, then $\tau_T(Z)=T$), $\epsilon(Z)=\varepsilon^+_{S,T}$ and so $\varphi_{T,u}(Z)=Z\exp(i\varepsilon^+_{S,T}W_{T,u})=\exp(i\varepsilon^+_{S,T}(W_u-\textrm{m}^+_{S,T}))$. As $Z\neq 1$, we necessarily have $W^+_{S,T}>0$. Thus if $\varepsilon^+_{S,T}=1$,
$$\tau_T(Z)=\inf\{r\geq T:\ W_r-\textrm{m}^+_{S,T}=0\ \textrm{or}\ \ l\}$$
and if $\varepsilon^+_{S,T}=-1$,
$$\tau_T(Z)=\inf\{r\geq T:\ W_r-\textrm{m}^+_{S,T}=0\ \textrm{or}\ \ 2\pi-l\}.$$
Since $u<\tau_T(Z)$, we have $\textrm{m}^+_{S,u}=\textrm{m}^+_{S,T}$ and $\varphi_{T,u}(Z)=\exp(i\varepsilon^+_{S,T}W^+_{S,u})$. On the other hand, since $u\leq\rho_S$,
$$\varphi_{S,u}(z)=\exp(i\varphi^+_{S,u}(0))=\exp(i\varepsilon^+_{S,u}W^+_{S,u}).$$
But $(S,T)\in\mathscr D^+$ (from $W^+_{S,T}>0$), $(S,u)\in\mathscr D^+$ (from $u<\tau_T(Z)$ which entails that $W^+_{S,u}>0$). Consequently $\varepsilon^+_{S,u}=\varepsilon^+_{S,T}$  and so $\varphi_{T,u}(Z)=\varphi_{S,u}(z)$. The case $\varphi_{S,\tau_S(z)}(z)=e^{il}$ can be done similarly.\\
\textbf{(iv)} Let $(z,u)\in E_{(iv)}$. Assume for example that $\varphi_{S,\tau_S(z)}(z)=1$ so that $W^+_{S,\tau_S(z)}=0$. Consider the first case: $\varepsilon^+_{S,T}=1$. Then $Z=e^{iW^+_{S,T}}$ and

$$\tau_T(Z)=\inf\{r\geq T:\ W_r-\textrm{m}^+_{S,T}\in\{0,l\}\}.$$

\noindent If $W_{\tau_T(Z)}-\textrm{m}^+_{S,T}=l$, then $u=\tau_T(Z)=\rho_S$ and $\varphi_{S,u}(z)=\varphi_{T,u}(Z)=e^{il}$.\\ If $W_{\tau_T(Z)}-\textrm{m}^+_{S,T}=0$, then $\varphi_{T,\tau_T(Z)}(Z)=1$ and $\varphi_{T,u}(Z)=\varphi_{T,u}(1)$.\\
Since $\varphi_{S,\tau_S(z)}(z)=1$, we have $\varphi_{S,u}(z)=\varphi_{S,u}(1)$. Moreover $W^+_{T,\tau_T(Z)}=W^+_{S,\tau_T(Z)}=0$, which implies $W^+_{T,u}=W^+_{S,u}$ (since $u\geq\tau_T(Z)$).\\
Now, if $u$ satisfies $W^+_{T,u}=W^+_{S,u}=0$, then $\varphi_{T,u}(Z)=\varphi_{S,u}(z)=1$. If not, $\textrm{m}^+_{T,u}=\textrm{m}^+_{S,u}$ and $(T,u), (S,u)$ are in $\mathscr D^+$. This implies $\varepsilon^+_{T,u}=\varepsilon^+_{S,u}$ and $\varphi_{T,u}(Z)=\varphi_{S,u}(z)$ exactly as in (ii).\\
Assume now that $\varepsilon^+_{S,T}=-1$, then $\tau_T(Z)$ satisfies $W_{\tau_T(Z)}-\textrm{m}^+_{S,T}=0$ (recall that $\tau_T(Z)\leq\rho_S$) and $\varphi_{T,u}(Z)=\varphi_{S,u}(z)$ as before.\\
The result for $K^{m^+,m^-}$ can be proved by replacing $\varphi_{S,T}(z)$ by $e^{iW^+_{S,T}}$ in $E_{(iii)}$ and $E_{(iv)}$. However, the proof remains similar.
\end{proof}
\begin{corl}\label{g}
Let $S\leq T$ be two finite $(\mathcal F^W_{-\infty,r})_{r\in\R}$-stopping times. Then, with probability 1, for all $u\geq T, z\in\mathscr C$, we have
$$\varphi_{S,u}(z)=\varphi_{T,u}\circ \varphi_{S,T}(z)$$
and
$$K^{m^+,m^-}_{S,u}(z)=K^{m^+,m^-}_{S,T}K^{m^+,m^-}_{T,u}(z).$$
\end{corl}
\begin{proof}
Fix $k\in\N$ and define the family of $(\mathcal F^W_{-\infty,r})_{r\in\R}$-stopping times $(T^i)_{i\geq 0}$ by $T^0=(T\vee \rho^k_S)\wedge \rho^{k+1}_S$ and $T^i=\rho_{T^{i-1}}$ for $i\geq 1$. As $r\longmapsto\rho_r$ is increasing, we have $\rho^{k+i}_S\leq T^i\leq \rho^{k+i+1}_S$ for all $i\geq 0$. Applying successively Proposition \ref{b}, we have a.s. for all $z\in\mathscr C$, $i\geq0$ and all $u\in[\rho^{k+i}_S,T^i]$ , $$\varphi_{S,u}(z)=\varphi_{\rho^{k+i}_S,u}\circ \varphi_{T^{i-1},\rho^{k+i}_S}\circ\cdots\circ \varphi_{T^0,\rho^{k+1}_S}\circ \varphi_{\rho^k_S,T^0}\circ\varphi_{S,\rho^k_S}(z)$$
and for all $u\in[T^i,\rho^{k+i+1}_S]$,
$$\varphi_{S,u}(z)=\varphi_{T^i,u}\circ\varphi_{\rho^{k+i}_S,T^i}\circ\cdots\circ \varphi_{T^0,\rho^{k+1}_S}\circ \varphi_{\rho^k_S,T^0}\circ\varphi_{S,\rho^k_S}(z).$$
On $\{\rho^k_S\leq T<\rho^{k+1}_S\}$, we have $T^i=\rho^i_T$ for all $i\geq0$ whence a.s. on $\{\rho^k_S\leq T<\rho^{k+1}_S\}$, for all $z\in\mathscr C$ and all $i\geq 0$,
$$\varphi_{S,u}(z)=\varphi_{\rho^{k+i}_S,u}\circ \varphi_{\rho^{i-1}_T,\rho^{k+i}_S}\circ\cdots\circ \varphi_{T,\rho^{k+1}_S}\circ \varphi_{S,T}(z)\ \textrm{for all}\ u\in[\rho^{k+i}_S,\rho_T^i]$$
and
$$\varphi_{S,u}(z)=\varphi_{\rho_T^i,u}\circ\varphi_{\rho^{k+i}_S,\rho_T^i}\circ\cdots\circ\varphi_{T,\rho^{k+1}_S}\circ \varphi_{S,T}(z)\ \textrm{for all}\ u\in[\rho^{i}_T,\rho_S^{k+i+1}].$$
Now define the family $(S^i)_{i\geq 1}$ of $(\mathcal F^W_{-\infty,r})_{r\in\R}$-stopping times by $S^1=(T\vee \rho^{k+1}_S)\wedge \rho^1_T$ and $S^{i+1}=\rho_{S^i}$ for $i\geq 1$. Then for all $i\geq 0$, $\rho^i_T\leq S^{i+1}\leq \rho^{i+1}_T$. Applying again Proposition \ref{b}, we get a.s. for all $z\in\mathscr C$, $i\geq 0$ and all $u\in[\rho^i_T,S^{i+1}]$,
$$\varphi_{T,u}(\varphi_{S,T}(z))=\varphi_{\rho^i_T,u}\circ\varphi_{S^i,\rho^i_T}\circ\cdots\circ\varphi_{S^1,\rho^1_T}\circ\varphi_{T,S^1}(\varphi_{S,T}(z))$$
and for all $u\in[S^{i+1},\rho^{i+1}_T]$,
$$\varphi_{T,u}(\varphi_{S,T}(z))=\varphi_{S^{i+1},u}\circ\varphi_{\rho^i_T,S^{i+1}}\circ\cdots\circ\varphi_{S^1,\rho^1_T}\circ\varphi_{T,S^1}(\varphi_{S,T}(z)).$$
On $\{\rho^k_S\leq T<\rho^{k+1}_S\}$, we have $S^i=\rho^{k+i}_S$ for all $i\geq 1$. Consequently a.s. on $\{\rho^k_S\leq T<\rho^{k+1}_S\}$, for all $z\in\mathscr C$, $i\geq 0$ and all $u\in[\rho^i_T,\rho_S^{k+i+1}]$,

$$\varphi_{T,u}(\varphi_{S,T}(z))=\varphi_{\rho^i_T,u}\circ\varphi_{\rho_S^{k+i},\rho^i_T}\circ\cdots\circ\varphi_{\rho_S^{k+1},\rho^1_T}\circ\varphi_{T,\rho_S^{k+1}}(\varphi_{S,T}(z))$$
and for all $u\in[\rho_S^{k+i+1},\rho^{i+1}_T]$.
$$\varphi_{T,u}(\varphi_{S,T}(z))=\varphi_{\rho_S^{k+i+1},u}\circ\varphi_{\rho^i_T,\rho_S^{k+i+1}}\circ\cdots\circ\varphi_{\rho_S^{k+1},\rho^1_T}\circ\varphi_{T,\rho_S^{k+1}}(\varphi_{S,T}(z)).$$
We have thus shown that a.s. for all $z\in\mathscr C$ and all $u\geq T$,
$$1_{\{\rho^k_S\leq T<\rho^{k+1}_S\}}\varphi_{T,u}\circ \varphi_{S,T}(z)=1_{\{\rho^k_S\leq T<\rho^{k+1}_S\}}\varphi_{S,u}(z).$$
By summing over $k$, we get that a.s. $\forall z\in\mathscr C$, $\forall u\geq T$, $\varphi_{T,u}\circ \varphi_{S,T}(z)=\varphi_{S,u}(z)$. The flow property for $K^{m^+,m^-}$ holds by the same reasoning.
\end{proof}
\subsection{$K^{m^+,m^-}$ can be obtained by filtering $\varphi$.\newline}
For all $-\infty\leq s\leq t\leq+\infty$, let $$\mathcal F^{U^+,U^-,W}_{s,t}=\sigma(U^+_{u,v},U^{-}_{u,v},W_{u,v};\; s\leq u\leq v\leq t)=\sigma(K^{+}_{u,v},K^{-}_{u,v};\; s\leq u\leq v\leq t).$$
Corollary \ref{g} entails the following
\begin{prop}\label{oo}
For all $z\in \mathscr C$, all $s<t$ and all continuous function $f$, a.s.
$$K^{m^+,m^-}_{s,t}f(z)=E\left[f(\varphi_{s,t}(z))\left| \mathcal F^{U^+,U^-,W}_{s,t}\right.\right].$$
\end{prop}
\begin{proof}
Fix $s\leq t, z\in\mathscr C$ and $f\in C(\mathscr C)$. Properties (ii) and (iii) of Section \ref{klp} imply that a.s.
$$K^{m^+,m^-}_{s,t}f(z)1_{\{s\leq t\leq\rho_s\}}=E\left[f(\varphi_{s,t}(z))\left| \mathcal F^{U^+,U^-,W}_{s,t}\right.\right]1_{\{s\leq t\leq\rho_s\}}\ .$$
Define $$\mathcal F^{\varepsilon^+,\varepsilon^-,U^+,U^-,W}_{s,t}=\sigma(\varepsilon^{\pm}_{u,v}, U^{\pm}_{u,v}, W_{u,v}; \;s\leq u\leq v\leq t)=\sigma(\varphi^{\pm}_{u,v}, K^{\pm}_{u,v}; \;s\leq u\leq v\leq t).$$
If $Z$ is a random variable independent of $\mathcal F^{\varepsilon^+,\varepsilon^-,U^+,U^-,W}_{s,t}$, then a.s.
\begin{equation}\label{kk}
K^{m^+,m^-}_{s,t}f(Z)1_{\{s\leq t\leq\rho_s\}}=E\left[f(\varphi_{s,t}(Z))\left| \mathcal F^{U^+,U^-,W}_{s,t}\right.\right]1_{\{s\leq t\leq\rho_s\}}\ .
\end{equation}
For $n\geq 1$ and $i\in[0,n]$, let $t^n_i=s+\frac{(t-s)i}{n}, A_{n,i}=\{t^n_i\leq\rho_{t^n_{i-1}}\}$ and for $n\geq 1$ let $A_n=\cap_{i=1}^{n}A_{n,i}$. Note that $A_{n,i}\in \mathcal F^{W}_{t^n_{i-1},t^n_i}$ and $A_n\in\mathcal F^{W}_{s,t}$. Then since $K^{\pm}$ and $\varphi^{\pm}$ are stochastic flows, $\mathcal F^{\varepsilon^+,\varepsilon^-,U^+,U^-,W}_{s,t}=\bigvee_{i=1}^{n} \mathcal \mathcal F^{\varepsilon^+,\varepsilon^-,U^+,U^-,W}_{t^n_{i-1},t^n_i}$. By Corollary \ref{g}, a.s.
$$K^{m^+,m^-}_{s,t}(z)=K^{m^+,m^-}_{s,t^n_1}\cdots K^{m^+,m^-}_{t^n_{n-1},t}(z)$$
and
$$\varphi_{s,t}(z)=\varphi_{t^n_{n-1},t}\circ\cdots\circ \varphi_{{s,t^n_1}}(z).$$
Recall that the $\sigma$-fields $\left(\mathcal F^{\varepsilon^+,\varepsilon^-,U^+,U^-,W}_{t^n_{i-1},t^n_i}\right)_{1\leq i\leq n}$ are independent. Then, using (\ref{kk}), we get that a.s.
$$K^{m^+,m^-}_{s,t}f(z)1_{A_n}=E\left[f(\varphi_{s,t}(z))\left| \mathcal F^{U^+,U^-,W}_{s,t}\right.\right]1_{A_n},$$
and therefore a.s.
$$K^{m^+,m^-}_{s,t}f(z)=E\left[f(\varphi_{s,t}(z))\left| \mathcal F^{U^+,U^-,W}_{s,t}\right.\right]1_{A_n}+\left(K^{m^+,m^-}_{s,t}f(z)-E\left[f(\varphi_{s,t}(z))\left| \mathcal F^{U^+,U^-,W}_{s,t}\right.\right]\right)1_{A_n^c}.$$
To finish the proof, it remains to prove that $\mathbb P(A_n^c)\rightarrow0$ as $n\rightarrow\infty$. Write
$$\mathbb P(A_n^c)\leq\sum_{i=1}^{n}\mathbb P(A_{n,i}^c)=\sum_{i=1}^{n}\mathbb P(t^n_i-t^n_{i-1}>\rho_{t^n_{i-1}}-t^n_{i-1})=n\mathbb P\left(\frac{t-s}{n}>\rho_0\right).$$
Let $\rho^{\pm}=\inf\{r\geq0: W_{0,r}^{\pm}=l\}$. Then $$\mathbb P(A_n^c)\leq n\left(\mathbb P\left(\frac{t-s}{n}>\rho^+\right)+\mathbb P\left(\frac{t-s}{n}>\rho^-\right)\right)=2n\mathbb P\left(\frac{t-s}{n}>\rho^+\right).$$
We have $\rho^+\overset{law}{=}\inf\{r\geq0: |W_r|=l\}$. Let $T_l=\inf\{r\geq0: W_r=l\}$, then
$$\mathbb P(A_n^c)\leq 4n\mathbb P\left(\frac{t-s}{n}>T_l\right)=4n \int_{\frac{t-s}{n}}^{+\infty} \frac{l}{\sqrt{2\pi x^3}}\exp(\frac{-l^2}{2x})dx$$
(see \cite{MR1725357} page 80). By the change of variable $v=nx$, the right hand side converges to $0$ as $n\rightarrow\infty$ which finishes the proof.
\end{proof}
\subsection{The $L^2$ continuity.\newline}
To conclude that $K^{m^+,m^-}$ and $\varphi$ are two stochastic flows, it remains to prove the following
\begin{prop}\label{mm}
For all $t\geq0$, $\theta\in[0,2\pi[$ and $f\in C(\mathscr C)$, we have
$$\lim_{z\rightarrow e^{i\theta}}E\left[\left(f(\varphi_{0,t}(z))-f(\varphi_{0,t}( e^{i\theta}))\right)^2\right]=\lim_{z\rightarrow e^{i\theta}}E\left[\left(K_{0,t}^{m^+,m^-}f(z)-K_{0,t}^{m^+,m^-}f(e^{i\theta})\right)^2\right]=0.$$
\end{prop}
\begin{proof}
By Jensen's inequality and Proposition \ref{oo}, it suffices to prove the result only for $\varphi$ and by the proof of Lemma 1.11 \cite{MR2060298} (see also Lemma 1 \cite{MR2835247}), this amounts to show that 
\begin{equation}\label{HR}
\lim_{z\rightarrow e^{i\theta}}\mathbb P\left(d(\varphi_{0,t}(z),\varphi_{0,t}(e^{i\theta}))>\eta\right)=0.
\end{equation}
where $t>0, \eta>0$ and $\theta\in[0,2\pi[$ are fixed from now on. For each $z\in\mathscr C$, let $A_{z}=\{d(\varphi_{0,t}(z),\varphi_{0,t}(e^{i\theta}))>\eta\}$ and denote $\tau_0(z)$ and $\varphi_{0,t}$ simply by $\tau(z)$ and $\varphi_{t}$. \\
\underline{First case} : $\theta=0$. For $\alpha\in]0,l[$, we have $\tau(e^{i\alpha})=\inf\{t\geq 0 : \alpha+W_t=0\ \textrm{or}\ l\}$ and
$$\mathbb P(A_{e^{i\alpha}})\leq\mathbb P(t<\tau(e^{i\alpha}))+\mathbb P\left(A_{e^{i\alpha}}\cap\{\varphi_{\tau(e^{i\alpha})}(e^{i\alpha})=1, t\geq\tau(e^{i\alpha})\}\right)+\mathbb P(\varphi_{\tau(e^{i\alpha})}(e^{i\alpha})=e^{il}).$$
If $t\geq\tau(e^{i\alpha})$ and $\varphi_{\tau(e^{i\alpha})}(e^{i\alpha})=1$, then $\varphi_{t}(e^{i\alpha})=\varphi_{t}(1)$, thus in the right-hand side, the second term equals $0$. Since $\lim_{\alpha\rightarrow 0+}\tau(e^{i\alpha})=0$ and $\mathbb P(\varphi_{\tau(e^{i\alpha})}(e^{i\alpha})=e^{il})=\mathbb P(\alpha+W_{\tau(e^{i\alpha})}=l)$, it is clear that  $\mathbb P(A_{e^{i\alpha}})\rightarrow0$ as $\alpha\rightarrow0+$ and similarly $\mathbb P(A_{e^{i\alpha}})\rightarrow0$ as $\alpha\rightarrow(2\pi)-$. Thus (\ref{HR}) holds for $\theta=0$ and by the same way for $\theta=l$.\\
\noindent \underline{Second case} : $\theta\in]l,2\pi[$. For all $\alpha\in ]l,2\pi[$, we have
\begin{eqnarray}
\mathbb P(A_{e^{i\alpha}})&\leq&\mathbb P\big(A_{e^{i\alpha}}\cap\{\varphi_{\tau(e^{i\alpha})}(e^{i\alpha})=\varphi_{\tau(e^{i\theta})}(e^{i\theta})=1\}\big)\nonumber\\
&+&\mathbb P\big( A_{e^{i\alpha}}\cap\{\varphi_{\tau(e^{i\alpha})}(e^{i\alpha})=\varphi_{\tau(e^{i\theta})}(e^{i\theta})=e^{il}\}\big)+\epsilon_{\alpha,\theta}\nonumber\
\end{eqnarray}
where $$\epsilon_{\alpha,\theta}=\mathbb P\big( \varphi_{\tau(e^{i\alpha})}(e^{i\alpha})=1, \varphi_{\tau(e^{i\theta})}(e^{i\theta})=e^{il}\big) +\mathbb P\big( \varphi_{\tau(e^{i\alpha})}(e^{i\alpha})=e^{il}, \varphi_{\tau(e^{i\theta})}(e^{i\theta})=1\big)$$
which converges to $0$ as $\alpha\rightarrow\theta$. Let us prove that $$\lim_{\alpha\rightarrow \theta}\mathbb P(B_{\alpha})=0\ \textrm{where}\ \ B_{\alpha}=A_{e^{i\alpha}}\cap\{\varphi_{\tau(e^{i\alpha})}(e^{i\alpha})=\varphi_{\tau(e^{i\theta})}(e^{i\theta})=1\}.$$
For $l<\alpha<\theta$, write
$$\mathbb P(B_{\alpha})=\mathbb P(B_{\alpha}\cap\{t\leq\tau(e^{i\theta})\})+\mathbb P(B_{\alpha}\cap\{\tau(e^{i\theta})<t<\tau(e^{i\alpha})\})+\mathbb P(B_{\alpha}\cap{\{t\geq\tau(e^{i\alpha})\}}).$$
Since $\varphi_{\cdot}(e^{i\alpha})$ and $\varphi_{\cdot}(e^{i\theta})$ move parallely until one of them hits $1$ or $e^{il}$, it comes that
$$\lim_{\alpha\rightarrow \theta-}\bigg(\mathbb P(B_{\alpha}\cap\{t\leq\tau(e^{i\theta})\})+\mathbb P(B_{\alpha}\cap\{\tau(e^{i\theta})<t<\tau(e^{i\alpha})\})\bigg)=0.$$
Now
\begin{eqnarray}
\mathbb P(B_{\alpha}\cap{\{t\geq\tau(e^{i\alpha})\}})&=&\mathbb P(B_{\alpha}\cap{\{\tau(e^{i\alpha})\leq t\wedge\rho_{\tau(e^{i\theta})}\}})+\mathbb P(B_{\alpha}\cap{\{\rho_{\tau(e^{i\theta})}<\tau(e^{i\alpha})\leq t \}})\nonumber\\
&\leq& \mathbb P(B_{\alpha}\cap{\{\tau(e^{i\alpha})\leq t\wedge\rho_{\tau(e^{i\theta})}\}})+\mathbb P(\rho_{\tau(e^{i\theta})}<\tau(e^{i\alpha})).\nonumber\
\end{eqnarray}
Obviously $\lim_{\alpha\rightarrow \theta} \mathbb P(\rho_{\tau(e^{i\theta})}<\tau(e^{i\alpha}))=\mathbb P(\rho_{\tau(e^{i\theta})}<\tau(e^{i\theta}))=0$. Set $Y=\varphi_{\tau(e^{i\theta})}(e^{i\alpha})$, then a.s. on $B_{\alpha}\cap{\{\tau(e^{i\alpha})\leq t\wedge\rho_{\tau(e^{i\theta})}\}}$, we have $\varphi_{t}(e^{i\alpha})=\varphi_{\tau(e^{i\theta}),t}(Y)$ by Corollary \ref{g} and $\tau_{\tau(e^{i\theta})}(Y)=\tau(e^{i\alpha})\leq \rho_{\tau(e^{i\theta})}$. Recall that $\varphi_{\tau(e^{i\theta}),s}(Y):=\varphi_{\tau(e^{i\theta}),s}(1)$ for all $s\in[\tau_{\tau(e^{i\theta})}(Y),\rho_{\tau(e^{i\theta})}]$ and consequently $\varphi_{\tau(e^{i\theta}),s}(Y)=\varphi_{\tau(e^{i\theta}),s}(1)$ for all $s\geq\tau_{\tau(e^{i\theta})}(Y)$ (by the definition of $\varphi$). This shows that a.s. on $B_{\alpha}\cap{\{\tau(e^{i\alpha})\leq t\wedge\rho_{\tau(e^{i\theta})}\}}$, we have $$d(\varphi_{t}(e^{i\theta}),\varphi_{t}(e^{i\alpha}))=d(\varphi_{\tau(e^{i\theta}),t}(1),\varphi_{\tau(e^{i\theta}),t}(1))=0.$$
Finally $\lim_{\alpha\rightarrow \theta-}\mathbb P(B_{\alpha})=0$ and by interchanging the roles of $\theta$ and $\alpha$, we have $\lim_{\alpha\rightarrow \theta+}\mathbb P(B_{\alpha})=0$. Similarly $$\lim_{\theta\rightarrow \theta}\mathbb P\left(A_{e^{i\alpha}}\cap\{\varphi_{\tau(e^{i\theta})}(e^{i\theta})=\varphi_{\tau(e^{i\alpha})}(e^{i\alpha})=e^{il}\}\right)=0$$
so that (\ref{HR}) is satisfied for all $\theta\in]l,2\pi[$. By the same way, it is also satisfied for all $\theta\in]0,l[$.                                                                             
\end{proof}
\subsection{The flows $\varphi$ and $K^{m^+,m^-}$ solve $(T_{\mathscr C})$.\newline}
In this paragraph we prove the following
\begin{prop}
Both $\varphi$ and $K^{m^+,m^-}$ solve $(T_{\mathscr C})$.
\end{prop}
\begin{proof}
First we check the result for $\varphi$. We will denote $\varphi^{\pm}_{s,t}(0)$ simply by $\varphi^{\pm}_{s,t}$ and the mapping $z\mapsto\varphi^{\pm}_{s,t}(z)$ by $(\varphi^{\pm}_{s,t}(z))_{z\in\mathscr C}$ to avoid any confusion. An important consequence of the modifications defined in Section \ref{chayynt} which is the key argument here is that $\varphi^{\pm}_{S,S+\cdot}$ is a Brownian motion for any finite $(\mathcal F^W_{0,\cdot})$-stopping time $S$. To justify this, consider a finite $(\mathcal F^W_{0,\cdot})$-stopping time $S$ and for $q\geq 1$ and $t>0$, set
$$S_q=\frac{\lfloor qS\rfloor+1}{q},\ S_{q,t}=S_q-\frac{2}{q}+t.$$
Let $t>0$, then a.s. $(S,S+t)\in\mathscr D^+$ and for $q$ large enough, we have $(S_q,S_{q,t})\in\mathscr D^+$ and $\textrm{m}^+_{S_q,S_{q,t}}=\textrm{m}^+_{S,S+{t}}$. Lemma \ref{aa} (ii) implies that a.s. for $q$ large enough $\varepsilon^+_{S,S+t}=\varepsilon^+_{S_q,S_{q,t}}$. Thus a.s.
\begin{equation}\label{lop}
\varphi^+_{S,S+t}=\lim_{q\rightarrow\infty}\varphi^+_{S_q,S_{q,t}}.
\end{equation}
Let $0<t_1<\cdots<t_n$ and take a family $(f_i)_{1\leq i\leq n}$ of bounded continuous functions from $\R$ into $\R$. Using the independence of increments and the stationarity of $\varphi^+$, we have
\begin{eqnarray}
E\left[\prod_{i=1}^{n}f_i(\varphi^+_{S,S+t_i})\right]&=&\lim_{q\rightarrow\infty}E\left[\prod_{i=1}^{n}f_i(\varphi^+_{S_q,S_{q,t_i}})\right]\nonumber\\
&=&\lim_{q\rightarrow\infty}\sum_{h\in\N}E\left[\prod_{i=1}^{n}f_i(\varphi^+_{\frac{h+1}{q},\frac{h-1}{q}+t_i})1_{\{\frac{h}{q}\leq S<\frac{h+1}{q}\}}\right]\nonumber\\
&=&\lim_{q\rightarrow\infty}\sum_{h\in\N}E\left[\prod_{i=1}^{n}f_i(\varphi^+_{\frac{h+1}{q},\frac{h-1}{q}+t_i})\right]\mathbb P\left(\frac{h}{q}\leq S<\frac{h+1}{q}\right)\nonumber\\
&=&\lim_{q\rightarrow\infty}\sum_{h\in\N}E\left[\prod_{i=1}^{n}f_i(\varphi^+_{0,t_i-\frac{2}{q}})\right]\mathbb P\left(\frac{h}{q}\leq S<\frac{h+1}{q}\right)\nonumber\\
&=&\lim_{q\rightarrow\infty}E\left[\prod_{i=1}^{n}f_i(\varphi^+_{0,t_i-\frac{2}{q}})\right]\nonumber\\
&=&E\left[\prod_{i=1}^{n}f_i(\varphi^+_{0,t_i})\right].\nonumber\
\end{eqnarray}
Since $\varphi^+_{0,\cdot}$ is a Brownian motion, the same holds for $\varphi^+_{S,S+\cdot}$. Now the rest  of the proof will be divided into three steps.\\
\textbf{First step.} Let $S$ be a finite $(\mathcal F^W_{0,\cdot})$-stopping time. Then for all $z\in\mathscr C, f\in C^2(\mathscr C)$, a.s. $\forall t\in[0,\rho_S-S]$,
$$f(\varphi_{S,S+t}(z))=f(z)+\int_0^t (f'\epsilon)(\varphi_{S,S+u}(z))dW_{S,S+u}+ \frac{1}{2}\int_0^tf''(\varphi_{S,S+u}(z))du.$$
We first prove this for $z=1$. By It\^o's formula, for all $f\in C^2(\mathscr C)$ a.s. $\forall t\geq0$,
$$f(\exp(i\varphi^+_{S,S+t}))=f(1)+\int_0^t f'(\exp(i\varphi^+_{S,S+u}))d\varphi^+_{S,S+u} + \frac{1}{2}\int_0^tf''(\exp(i\varphi^+_{S,S+u}))du.$$
Tanaka's formula for local time yields  a.s. $\forall t\in[0,\rho_S-S]$,
$$|\varphi^+_{S,S+t}|=\int_{0}^{t}\textrm{sgn}(\varphi^+_{S,S+u})d\varphi^+_{S,S+u}+L_t$$
where $L_t$ is the local time in $0$ of $\varphi^{+}_{S,S+\cdot}$. By construction, $|\varphi^+_{S,S+t}|=W^{+}_{S,S+t}$ for all $t$. So we can deduce from the previous line that a.s. $\forall t\in[0,\rho_S-S]$,
$$\int_{0}^{t}\textrm{sgn}(\varphi^+_{S,S+u})d\varphi^+_{S,S+u}+L_t=W^{+}_{S,S+t}.$$
Thus by unicity of the Doob-Meyer decomposition, a.s. $\forall t\in[0,\rho_S-S]$,
$$\int_{0}^{t}\textrm{sgn}(\varphi^+_{S,S+u})d\varphi^+_{S,S+u}=W_{S,S+t}.$$
Since $\textrm{sgn}(\varphi^+_{S,S+u})=\varepsilon^+_{S,S+u}$ a.s., we get a.s. $\forall t\in[0,\rho_S-S]$, $$\varphi^+_{S,S+t}=\int_{0}^{t}\varepsilon^+_{S,S+u}dW_{S,S+u}=\int_{0}^{t}\epsilon(\varphi_{S,S+u}(1))dW_{S,S+u}.$$
Recall that $\varphi_{S,S+t}(1)=e^{i\varphi^+_{S,S+t}}$ for all $t\in[0,\rho_S-S]$, thus the first step holds for $z=1$. The first step is similarly satisfied for $z=e^{il}$ and for all $z\in\mathscr C\setminus\{1,e^{il}\}$ by distinguishing the cases $t\leq \tau_{S}(z)-S$ and $t>\tau_{S}(z)-S$.\\
\textbf{Second step.} Let $S$ be a finite $(\mathcal F^W_{0,\cdot})$-stopping time, $\mathcal G_t=\sigma(\varphi_{0,u}(z), z\in\mathscr C, 0\leq u\leq t)$, $t\geq0$. Then $\sigma(\varphi_{S,(S+u)\wedge \rho_S}(z), z\in\mathscr C, u\geq0)$ is independent of $\mathcal G_{S}$.\\
Clearly
$$\sigma(\varphi_{S,(S+u)\wedge \rho_S}(z), z\in\mathscr C, u\geq0)\subset \sigma(\varphi^+_{S,S+u}, u\geq0)\vee \sigma(\varphi^-_{S,S+u}, u\geq0).$$
Fix $0<u_1<\cdots<u_n$, then a.s. $(S,S+u_1),\cdots,(S,S+u_n)$ are in $\mathscr D^+\cap \mathscr D^-$. Take a family $\{f_1,g_1,\cdots,f_n,g_n\}$ of bounded continuous functions from $\R$ into $\R$ and let $A\in \mathcal G_{S}$. By (\ref{lop}), we have
$$E\bigg[\prod_{i=1}^{n}f_i(\varphi^+_{S,S+u_i})g_i(\varphi^-_{S,S+u_i})1_{A}\bigg]=\lim_{q\rightarrow\infty}E\bigg[\prod_{i=1}^{n}f_i(\varphi^+_{S_q,S_{q,u_i}})g_i(\varphi^-_{S_q,S_{q,u_i}})1_A\bigg].$$
For $q$ large enough $(\frac{2}{q}<u_1)$, we have
$$E\bigg[\displaystyle\prod_{i=1}^{n}f_i(\varphi^+_{S_q,S_{q,u_i}})g_i(\varphi^-_{S_q,S_{q,u_i}})1_A\bigg]$$
$$=\displaystyle\sum_{m\geq0}E\bigg[\prod_{i=1}^{n}f_i\big(\varphi^+_{\frac{m+1}{q},\frac{m-1}{q}+u_i}\big) g_i\big(\varphi^-_{\frac{m+1}{q},\frac{m-1}{q}+u_i}\big)1_{A\cap \{\frac{m}{q}\leq S <\frac{m+1}{q}\}}\bigg]$$
with $A\cap \{\frac{m}{q}\leq S<\frac{m+1}{q}\}\in\mathcal G_{\frac{m+1}{q}}\subset \sigma(\varphi^+_{u,v}(z),\varphi^-_{u,v}(z), z\in\mathscr C, 0\leq u\leq v\leq \frac{m+1}{q})$. Now using the independence of increments and the stationarity of $(\varphi^+,\varphi^-)$, the second step easily holds.\\
\textbf{Third step}. $\varphi$ solves $(T_{\mathscr C})$.\\
Denote $\rho^k_0$ simply by $\rho^k$. Then a.s. for all $k\in\N$ and $z\in\mathscr C$, $u\longmapsto\varphi_{\rho^k,u}(z)$ is continuous on $[\rho^k,\rho^{k+1}]$. Consequently for all $z\in\mathscr C$, a.s. $u\longmapsto\varphi_{0,u}(z)$ is continuous on $[0,+\infty[$ and in particular, $\varphi_{0,\rho^k}(z)$ is $\mathcal G_{\rho^k}$ measurable. Now fix $f\in C^2(\mathscr C), t\geq0, z\in\mathscr C$ and define for all $y\in\mathscr C$,
\begin{eqnarray}
H_{(f,t)}(y)&=&f(\varphi_{\rho^1,\rho^1+t\wedge (\rho^{2}-\rho^1)}(y))-f(y)-\int_0^{t\wedge (\rho^{2}-\rho^1)} (f'\epsilon)(\varphi_{\rho^1,\rho^1+u}(y))dW_{\rho^1,\rho^1+u}\nonumber\\
&-& \frac{1}{2}\int_0^{t\wedge (\rho^{2}-\rho^1)}f''(\varphi_{\rho^1,\rho^1+u}(y))du.\nonumber\
\end{eqnarray}
Then a.s. $y\longmapsto H_{(f,t)}(y)$ is measurable from $\mathscr C$ into $\R$. Moreover $H_{(f,t)}$ is $\sigma(\varphi_{\rho^1,(\rho^1+u)\wedge \rho^{2}}(z), u\geq0, z\in\mathscr C)$-measurable and $H_{(f,t)}(y)=0$ a.s. for all $y\in\mathscr C$ by the first step. The second step yields $H_{(f,t)}(\varphi_{0,\rho^1}(z))=0$ a.s. and we may replace $y$ by $\varphi_{0,\rho^1}(z)$ directly in the stochastic integral so that, using the flow property, we get
\begin{eqnarray}
f\left(\varphi_{0,\rho^1+t\wedge (\rho^{2}-\rho^1)}(z)\right)&=&f(\varphi_{0,\rho^1}(z))+\int_0^{t\wedge (\rho^{2}-\rho^1)} (f'\epsilon)(\varphi_{0,\rho^1+u}(z))dW_{\rho^1,\rho^1+u}\nonumber\\
&+& \frac{1}{2}\int_0^{t\wedge (\rho^{2}-\rho^1)}f''(\varphi_{0,\rho^1+u}(z))du\nonumber\\
&=&f(z)+\displaystyle\int_0^{\rho^1+t\wedge (\rho^{2}-\rho^1)} \bigg((f'\epsilon)(\varphi_{0,u}(z))dW_{u}+ \frac{1}{2}f''(\varphi_{0,u}(z))\bigg)du.\nonumber
\end{eqnarray}
By induction, we have a.s. $\forall k\in\mathbb N,$
\begin{eqnarray}
f(\varphi_{0,\rho^k+t\wedge (\rho^{k+1}-\rho^k)}(z))&=&f(z)+\int_0^{\rho^k+t\wedge (\rho^{k+1}-\rho^k)} (f'\epsilon)(\varphi_{0,u}(z))dW_{u}\nonumber\\
&+& \frac{1}{2}\int_0^{\rho^k+t\wedge (\rho^{k+1}-\rho^k)}f''(\varphi_{0,u}(z))du.\nonumber\
\end{eqnarray}
This implies that $\varphi$ solves $(T_{\mathscr C})$. The fact that $K^{m^+,m^-}$ solves $(T_{\mathscr C})$ is similar to Proposition 4.1 (ii) in \cite{MR2235172} using Proposition \ref{oo}.
\end{proof}
\section{Flows solutions of $(T_{\mathscr C})$}
From now on $(K,W)$ is a solution of $(T_{\mathscr C})$ defined on a probability space $(\Omega,\mathcal A,\mathbb P)$. 
Fix $s\in\R$ and $z\in \mathscr C$, then $(K_{s,t}(z))_{t\geq s}$ can be modified such that, a.s. the mapping $t\longmapsto K_{s,t}(z)$ is continuous from $[s,+\infty[$ into $\mathcal P(\mathscr C)$. It is the version we consider henceforth for all fixed $s$ and $z$.
\begin{lemma}\label{rz} (i) For all $z\in \mathscr C$ and $s\in\R$, denote $\tau_{s}(z)=\inf\{r\geq s,\ ze^{i\epsilon(z)W_{s,r}}=1\ \textrm{or}\ e^{il}\}$. Then a.s. $$K_{s,t}(z)=\delta_{ze^{i\epsilon(z)W_{s,t}}}, \ \textrm{if}\ \ s\leq t\leq \tau_{s}(z).$$
(ii) $\sigma(W)\subset\sigma(K)$.
 \end{lemma}
\noindent\begin{proof}
\textbf{(i)} We follow Lemma 3.1 \cite{MR2235172}. Define
\begin{equation}\label{geor}
\mathscr C^+=\{z\in\mathscr C: \arg(z)\in ]0,l[\} \quad \hbox{ and } \quad \mathscr C^-=\mathscr C\setminus\mathscr C^+.
\end{equation}
Fix $z\in\mathscr C^+$ and let
$$\tilde \tau_z=\inf\left\{t\geq 0 : K_{0,t}(z,\mathscr C^-)>0\right\}.$$
Let $f\in C^2(\mathscr C)$ such that $f(y)=\arg(y)$ if $y\in\mathscr C^+$. By applying $f$ in $(T_{\mathscr C})$, we have for $t<\tilde \tau_z$,
\begin{equation}\label{zer}
\int_{\mathscr C} \arg(y)K_{0,t}(z,dy)=\arg(z)+W_t.
\end{equation}
By applying $f^2$ in $(T_{\mathscr C})$ and using (\ref{zer}), we also have for $t<\tilde \tau_z$,
\begin{eqnarray}
K_{0,t}f^2(z)&=&f^2(z)+2\int_{0}^{t}\int_{\mathscr C}\arg(y)K_{0,u}(z,dy)dW_u+t\nonumber\\
&=&f^2(z)+2\int_{0}^{t}(\arg(z)+W_u)dW_u+t\nonumber\\
&=& (\arg(z)+W_t)^2.\nonumber
\end{eqnarray}
Thus that for $t<\tilde \tau_z$,
$$\int_{\mathscr C}(\arg(y)-\arg(z)-W_t)^2K_{0,t}(z,dy)=K_{0,t}f^2(z)-2(\arg(z)+W_t)K_{0,t}f(z)+(\arg(z)+W_t)^2=0.$$
By continuity a.s.
$$K_{0,t}(z)=\delta_{ze^{i\epsilon(z)W_{t}}}\ \textrm{for all}\ t\in[0,\tilde\tau_z].$$
The fact that $\tau_0(z)=\tilde\tau_z$ easily follows.\\
\textbf{(ii)} Let $(f_n)_{n\geq 1}$ be a sequence in $C^2(\mathscr C)$ such that $f'_n(z)\rightarrow \epsilon(z)$ as $n\rightarrow\infty$ for all $z\in\mathscr C\setminus\{1,e^{il}\}$. Applying $f_n$ in $(T_{\mathscr C})$, we get
$$\int_0^tK_{0,u}(\epsilon f_n')(1)dW_u=K_{0,t}f_n(1)-f_n(1)-\frac{1}{2}\int_0^tK_{0,u}f_n''(1)du.$$
It is easy to check that $\int_0^tK_{0,u}(\epsilon f_n')(1)dW_u$ converges towards $W_t$ in $L^2(\mathbb P)$ as $n\rightarrow\infty$ whence in $L^2(\mathbb P)$
$$W_t=\lim_{n\rightarrow\infty}\left(K_{0,t}f_n(1)-f_n(1)-\frac{1}{2}\int_0^tK_{0,u}f_n''(1)du\right)$$
which proves (ii).
\end{proof}
\subsection{Unicity of the Wiener solution.\newline}
Our aim in this section is to prove that $(T_{\mathscr C})$ admits only one Wiener solution (i.e. such that $\sigma(W)\subset\sigma(K)$). This solution is $K^{m^+,m^-}$ with $m^+=m^-=\delta_{\frac{1}{2}}$. For this, we will essentially follow the general idea of \cite{MR1905858}: the Wiener solution is unique because its Wiener chaos decomposition can be given (see (\ref{carref}) and (\ref{ghad}) below). Let $p$ be semigroup of the standard Brownian motion on $\R$. Then the semigroup of the Brownian motion on $\mathscr C$ writes
$$P_t(e^{ix},e^{iy})=\sum_{k\in\Z} p_t(x,y+2k\pi),\ \  x,y\in[0,2\pi[.$$
For all $f\in C^1(\mathscr C)$, we easily check that $P_tf\in C^1(\mathscr C)$ and $(P_tf)'=P_tf'$. Let $Af=\frac{1}{2}f'', f\in C^2(\mathscr C)$ be the generator of $P$.
\begin{prop}
Equation $(T_{\mathscr C})$ has at most one Wiener solution: If $(K,W)$ is a solution such that  $\sigma(W)\subset\sigma(K)$, then $\forall t\geq 0, f\in C^{\infty}(\mathscr C)$ and all $z\in\mathscr C$,
\begin{equation}\label{carref}
K_{0,t}f(z)=P_tf(z)+\sum_{n=1}^{\infty}J^n_tf(z)\ \textrm{in}\ L^2(\mathbb P)
\end{equation}
where
\begin{equation}\label{ghad}
J^n_tf(z)=\int_{0<s_1<\cdots<s_n<t}P_{s_1}(D(P_{s_2-s_1}\cdots D(P_{t-s_n}f)))(z)dW_{0,s_1}\cdots dW_{0,s_n}
\end{equation}
no longer depends on $K$ and $Df(z)=\epsilon(z) f'(z)$.
\end{prop}
\begin{proof}  Let $(K,W)$ be a solution of $(T_{\mathscr C})$ (not necessarily a Wiener flow). Our first aim is to establish the following
\begin{lemma}\label{kwir}
Fix $f\in C^{\infty}(\mathscr C)$ and $z\in\mathscr C$. Then
$$K_{0,t}f(z)=P_{t}f(z)+\int_{0}^{t}K_{0,u}(D(P_{t-u}f))(z)dW_u.$$
\end{lemma}
\begin{proof}
\noindent Let $f\in C^{\infty}(\mathscr C), z\in\mathscr C$ and denote $K_{0,t}$ simply by $K_{t}$. Note that the stochastic integral in the right-hand side is well defined:
$$\int_{0}^{t}E[K_{u}(D(P_{t-u}f))(z)]^2du\leq \int_{0}^{t}P_u((D(P_{t-u}f))^2)(z)du\leq \int_{0}^{t}||(P_{t-u}f)'||^2_{\infty}du$$
and the right-hand side is smaller than $t||f'||^2_{\infty}$. Now
$$K_{t}f(z)-P_{t}f(z)-\int_{0}^{t}K_{u}(D(P_{t-u}f))(z)dW_u=\displaystyle\sum_{p=0}^{n-1}(K_{\frac{{(p+1)t}}{n}}P_{t-\frac{{(p+1)t}}{n}}f-K_{\frac{{p}t}{n}}P_{t-\frac{{p}t}{n}}f)(z)$$
$$-\displaystyle\sum_{p=0}^{n-1}\int_{\frac{pt}{n}}^{\frac{(p+1)t}{n}}K_{u}D((P_{t-u}-P_{t-\frac{(p+1)t}{n}})f)(z)dW_u-\displaystyle\sum_{p=0}^{n-1}\int_{\frac{pt}{n}}^{\frac{(p+1)t}{n}}K_{u}D(P_{t-\frac{(p+1)t}{n}}f)(z)dW_u.$$
For all $p\in\{0,..,n-1\}$, set $f_{p,n}=P_{t-\frac{(p+1)t}{n}}f\in C^{\infty}(\mathscr C)$ and so by replacing $f$ by $f_{p,n}$ in $(T_{\mathscr C})$, we get
$$\int_{\frac{pt}{n}}^{\frac{(p+1)t}{n}}K_{u}(Df_{p,n})(z)dW_u=K_{\frac{(p+1)t}{n}}f_{p,n}(z)-K_{\frac{pt}{n}}f_{p,n}(z)-\int_{\frac{pt}{n}}^{\frac{(p+1)t}{n}}K_{u}(Af_{p,n})(z)du$$
$$=K_{\frac{(p+1)t}{n}}f_{p,n}(z)-K_{\frac{pt}{n}}f_{p,n}(z)-\frac{t}{n}K_{\frac{pt}{n}}(Af_{p,n})(z)-\int_{\frac{pt}{n}}^{\frac{(p+1)t}{n}}(K_{u}-K_{\frac{pt}{n}})(Af_{p,n})(z)du.$$
Then we can write
$$K_{t}f(z)-P_{t}f(z)-\int_{0}^{t}K_{u}(D(P_{t-u}f))(z)dW_u=A_1(n)+A_2(n)+A_3(n),$$
where
\begin{eqnarray}
A_1(n)&=&-\displaystyle\sum_{p=0}^{n-1}K_{\frac{pt}{n}}[P_{t-\frac{pt}{n}}f-P_{t-\frac{(p+1)t}{n}}f-\frac{t}{n}AP_{t-\frac{(p+1)t}{n}}f](z),\nonumber\\
A_2(n)&=&-\displaystyle\sum_{p=0}^{n-1}\int_{\frac{pt}{n}}^{\frac{(p+1)t}{n}}K_{u}D((P_{t-u}-P_{t-\frac{(p+1)t}{n}})f)(z)dW_u,\nonumber\\
A_3(n)&=&\displaystyle\sum_{p=0}^{n-1}\int_{\frac{pt}{n}}^{\frac{(p+1)t}{n}}(K_{u}-K_{\frac{pt}{n}})AP_{t-\frac{(p+1)t}{n}}f(z)du.\nonumber\
\end{eqnarray}
 Using $||K_{u}g||_{\infty}\leq ||g||_{\infty}$ for $g$ a bounded measurable function, we see that $|A_1(n)|$ is less than
$$\displaystyle\sum_{p=0}^{n-1}\left|\left|P_{t-\frac{(p+1)t}{n}}[P_{\frac{t}{n}}f-f-\frac{t}{n} Af]\right|\right|_{\infty}
\leq n\left|\left|P_{\frac{t}{n}}f-f-\frac{t}{n}Af\right|\right|_{\infty}=t\left|\left|\frac{P_{\frac{t}{n}}f-f}{\frac{t}{n}}-Af\right|\right|_{\infty}.$$
Since $f\in C^{\infty}(\mathscr C)$, this shows that $A_1(n)$ converges to $0$ as $n\rightarrow\infty$. Note that $A_2(n)$ is the sum of orthogonal terms in $L^2(\mathbb P)$. Consequently
$$||A_2(n)||^2_{L^2(\mathbb P)}=\displaystyle\sum_{p=0}^{n-1}\left|\left|\int_{\frac{pt}{n}}^{\frac{(p+1)t}{n}}K_{u}D((P_{t-u}-P_{t-\frac{(p+1)t}{n}})f)(z)dW_u\right|\right|^2_{L^2(\mathbb P)}.$$
By applying Jensen's inequality, we arrive at $$||A_2(n)||^2_{L^2(\mathbb P)}\leq\displaystyle\sum_{p=0}^{n-1}\int_{\frac{pt}{n}}^{\frac{(p+1)t}{n}}P_{u}V^2_u(z)du$$
where $V_u=(P_{t-u}f)'-(P_{t-\frac{(p+1)t}{n}}f)'=P_{t-u}f'-P_{t-\frac{(p+1)t}{n}}f'$. For all $u\in[\frac{pt}{n},\frac{(p+1)t}{n}]$, we have
$$P_{u}V^2_u(z)\leq ||V_u||^2_{\infty}=\left|\left|P_{t-\frac{(p+1)t}{n}}\left(P_{\frac{(p+1)t}{n}-u}f'-f'\right)\right|\right|^2_{\infty}\leq ||P_{\frac{(p+1)t}{n}-u}f'-f'||^2_{\infty}.$$
Consequently
$$||A_2(n)||^2_{L^2(\mathbb P)}\leq\displaystyle\sum_{p=0}^{n-1}\int_{\frac{pt}{n}}^{\frac{(p+1)t}{n}}||P_{\frac{(p+1)t}{n}-u}f'-f'||^2_{\infty}du=n\int_{0}^{\frac{t}{n}}||P_{u}f'-f'||^2_{\infty}du,$$
and one can deduce that $A_2(n)$ tends to 0 as $n\rightarrow+\infty$ in $L^2(\mathbb P)$. Now

$$||A_3(n)||_{L^2(\mathbb P)}\leq \displaystyle\sum_{p=0}^{n-1}\left|\left|\int_{\frac{pt}{n}}^{\frac{(p+1)t}{n}}(K_{u}-K_{\frac{pt}{n}})AP_{t-\frac{(p+1)t}{n}}f(z)du\right|\right|_{L^2(\mathbb P)}.$$
Set $h_{p,n}=AP_{t-\frac{(p+1)t}{n}}f$. Then $h_{p,n}\in C^{\infty}(\mathscr C)$ for all $p\in[0,n-1]$. By the Cauchy-Schwarz inequality

$$||A_3(n)||_{L^2(\mathbb P)}\leq\sqrt{t}\left\{\sum_{p=0}^{n-1}\int_{\frac{pt}{n}}^{\frac{(p+1)t}{n}}E[((K_{u}-K_{\frac{pt}{n}})h_{p,n}(z))^2]du\right\}^{\frac{1}{2}}.$$
If $u\in [\frac{pt}{n},\frac{(p+1)t}{n}]$:
\begin{eqnarray}
E[((K_{u}-K_{\frac{pt}{n}})h_{p,n}(z))^2]&\leq& E[K_{\frac{pt}{n}}(K_{\frac{pt}{n},u}h_{p,n}-h_{p,n})^2(z)]\nonumber\
\\
&\leq&E[K_{\frac{pt}{n}}(K_{\frac{pt}{n},u}h_{p,n}^2-2h_{p,n}K_{\frac{pt}{n},u}h_{p,n}+h_{p,n}^2)(z)]\nonumber\
\\
&\leq&P_{\frac{pt}{n}}\left(P_{u-\frac{pt}{n}}h_{p,n}^2-2h_{p,n}P_{u-\frac{pt}{n}}h_{p,n}+h_{p,n}^2\right)(z)\nonumber\
\\
&\leq&||P_{u-\frac{pt}{n}}h_{p,n}^2-2h_{p,n}P_{u-\frac{pt}{n}}h_{p,n}+h_{p,n}^2||_{\infty}\nonumber\
\\
&\leq&2||h_{p,n}||_{\infty}||P_{u-\frac{pt}{n}}h_{p,n}-h_{p,n}||_{\infty}+||P_{u-\frac{pt}{n}}h_{p,n}^2-h_{p,n}^2||_{\infty}.\nonumber\
\end{eqnarray}
Therefore $||A_3(n)||_{L^2(\mathbb P)}\leq \sqrt{t}(2C_1(n)+C_2(n))^{\frac{1}{2}}$, where
$$C_1(n)=\displaystyle\sum_{p=0}^{n-1}||h_{p,n}||_{\infty}\int_{\frac{pt}{n}}^{\frac{(p+1)t}{n}}||P_{u-\frac{pt}{n}}h_{p,n}-h_{p,n}||_{\infty}du$$
and $$C_2(n)=\displaystyle\sum_{p=0}^{n-1}\int_{\frac{pt}{n}}^{\frac{(p+1)t}{n}}||P_{u-\frac{pt}{n}}h_{p,n}^2-h_{p,n}^2||_{\infty}du.$$
From $||h_{p,n}||_{\infty}\leq ||Af||_{\infty}$ and $||P_{u-\frac{pt}{n}}h_{p,n}-h_{p,n}||_{\infty}\leq ||P_{u-\frac{pt}{n}}Af-Af||_{\infty}$, we get
$$C_1(n)\leq||Af||_{\infty}\displaystyle\sum_{p=0}^{n-1}\int_{\frac{pt}{n}}^{\frac{(p+1)t}{n}}||P_{u-\frac{pt}{n}}Af-Af||_{\infty}du\leq||Af||_{\infty}\int_{0}^{t}||P_{\frac{s}{n}}Af-Af||_{\infty}ds.$$
As $Af\in C^{\infty}(\mathscr C)$, $C_1(n)$ tends to $0$ obviously. On the other hand, $h_{p,n}^2\in C^{\infty}(\mathscr C)$ and so $$C_2(n)=\frac{1}{n}\displaystyle\sum_{p=0}^{n-1}\int_{0}^{t}||P_{\frac{s}{n}}h_{p,n}^2-h_{p,n}^2||_{\infty}ds\leq \frac{1}{n}\displaystyle\sum_{p=0}^{n-1}\int_{0}^{t}\left(\int_{0}^{\frac{s}{n}}||A h_{p,n}^2||_{\infty}du\right)ds.$$
Now we easily verify that $h_{p,n}, h_{p,n}', h_{p,n}''$ are uniformly bounded with respect to $n$ and $0\leq p\leq n-1$. As a result $C_2(n)$ tends to $0$ as $n\rightarrow\infty$. This establishes Lemma \ref{kwir}.
\end{proof}
Assume that $(K,W)$ is a Wiener solution of $(T_{\mathscr C})$ and for $t\geq 0, f\in C^{\infty}(\mathscr C)$ and $z\in\mathscr C$, let $K_{0,t}f(z)=P_tf(z)+\sum_{n=1}^{\infty}J^n_tf(z)$
be the decomposition in Wiener chaos of $K_{0,t}f(z)$ in $L^2$ sense. By iterating the identity of Lemma \ref{kwir},  we see that for all $n\geq1$, $J^n_tf(z)$ is given by (\ref{ghad}).
\end{proof}
\noindent\textbf{Consequences:} Let $K^{W}$ be the unique Wiener solution of $(T_{\mathscr C})$. Since $\sigma(W)\subset\sigma (K)$, we can define $K^{*}$ the stochastic flow obtained by filtering $K$ with respect to $\sigma(W)$ (Lemma 3-2 (ii) in \cite{MR2060298}). Then, for all $s\leq t$ and all $z\in \mathscr C$, a.s.
$$K_{s,t}^{*}(z)=E[K_{s,t}(z)|\sigma(W)].$$
As a result, $(K^{*},W)$ solves also $(T_{\mathscr C})$ and by the last proposition, for all $s\leq t$ and all $z\in \mathscr C$, a.s.
\begin{equation}\label{k}
E[K_{s,t}(z)|\sigma(W)]=K_{s,t}^{W}(z).
\end{equation}
\subsection{Proof of Theorem \ref{a} (2).\newline}
Using the flow property and the independence of increments satisfied by $K$, it is easily seen that the law of $(K_{0,t_1},\cdots,K_{0,t_n})$ for all $(t_1,\cdots,t_n)\in(\R_+)^n$ and therefore the law of $K$ is uniquely determined by the knowledge of the law of $K_{0,t}$ for all $t\geq 0$. In the sequel, we will show the existence of two probability measures $m^+$ and $m^-$ on $[0,1]$ with mean $\frac{1}{2}$ such that for all $t\geq 0, K^{m^+,m^-}_{0,t}\overset{law}{=}K_{0,t}$ which will imply Part (2) of Theorem \ref{a}.
\subsubsection{A stochastic flow of mappings associated to $K$.\label{tss}\newline}
Let $P_t^{n}=E[K_{0,t}^{\otimes n}]$ be the consistent family of Feller semigroups associated to $K$. By Theorem 4.1 \cite{MR2060298}, a consistent family of coalescent Markovian semigroups $(P^{n,c})_{n\geq 1}$ is associated to $(P^{n})_{n\geq 1}$. The Feller process associated to $P^n$ (resp. to $P^{n,c}$) will be called the $n$-point motion of $P^n$ (resp. to $P^{n,c}$). The consistent family $(P^{n,c})_{n\geq 1}$ will be such that
\begin{itemize}
\item[(i)] The $n$-point motion of $P^{n,c}$ up to its entrance time in $\Delta_n$ is distributed as the $n$-point motion of $P^n$ up to its entrance time in $\Delta_n$, where $\Delta_n=\{x\in\mathscr C^n;\;  \exists i\neq j, \; x_i=x_j\}$.
\item[(ii)] The $n$-point motion $(X^1,\dots,X^n)$ of $P^{n,c}$ is such that if $X^i_0=X^j_0$ then for all $t>0$, $X^i_t=X^j_t$.
\end{itemize}
A possible construction of such a family is the following. Fix $(x^1,\cdots,x^n)\in \mathscr C ^n$ and let $X=(X^1,\dots,X^n)$ be the $n$ point motion started at $(x^1,\cdots,x^n)$ associated to $P^n$.  Let $$T_1=\inf\{t\geq 0, \exists i\neq j,\ X_t^i=X_t^j\}.$$
For $t\in [0,T_1]$, define $Y_t:=X_t$. Let $1\le i_1<\dots<i_k\le n$ be such that $\{Y_{T_1}^{i_j};\;1\le j\le k\}=\{Y_{T_1}^i;\;1\le i\le n\}$ and where $k=\hbox{Card}\{Y^i_{T_1};\;1\le i\le n\}$.
%Suppose that $Y_{T_1}^n(i)=Y_{T_1}^n(j)$ with $i<j$. 
Then define the process 
$$Z_t^i=X_t^{i_j}\quad \textrm{for } t\geq T_1
\hbox{ and when } Y_{T_1}^i=Y_{T_1}^{i_j}.$$
Now set $$T_2=\inf\{t\geq T_1, \exists j\neq l, \ Z_t^{i_j}=Z_t^{i_l}\}.$$ 
For $t\in [T_1,T_2]$, we define $Y_t=Z_t$ and so on. 

%%%%%%%%%%%%%%%FIGURE%%%%%%%%%%%%%%%%%%%%
%\begin{figure}[htbp]
%\begin{center}
%%%%%%%%%%%%\resizebox{\textwidth}{4cm}{\input{fig_intro1.pstex_t}}
%\resizebox{14cm}{4cm}{\input{fig_intro1.pstex_t}}
%\caption{The coalescent semigroup.}
%\end{center}
%\end{figure}
%%%%%%%%%%%%\resizebox{\textwidth}{4cm}{\input{fig_intro1.pstex_t}}

In this way, we construct a Markov process $Y$. It is the $n$ point motion of the family of semigroup $P^{n,c}$. 
Note that such a construction does not insure that these semigroups are fellerian.
\begin{lemma}\label{ww}
$(P^{n,c})_{n\geq 1}$ is a consistent family of coalescent Feller semigroups associated with a flow of mappings $\varphi^c$.
\end{lemma}
\begin{proof}
For each $(x,y)\in \mathscr C^2$, let $(X^x_t,Y^y_t)_{t\geq0}$ be the two point motion started at $(x,y)$ associated with $P^{2}$ constructed as in Section 2.6 \cite{MR2060298} on an extension $(\Omega\times \Omega',\mathcal E,\mathbb Q)$ of $(\Omega,\mathcal{A},\mathbb P)$ such that the law of $(X^x_t,Y^y_t)$ given $\omega\in\Omega$ is $K_{0,t}(x)\otimes K_{0,t}(y)$. Define
$$T^{x,y}:=\inf\{t\geq0 : X^x_t=Y^y_t\}.$$
By Theorem $4.1$ \cite{MR2060298}, we only need to check that: for all $t>0,\varepsilon>0$ and $x\in \mathscr C,$ $$\lim\limits_{\substack{y \to x}}\mathbb Q(\{T^{x,y}>t\}\cap\{d(X^x_t,Y^y_t)>\varepsilon\})=0 \ \ \ \ \ \ (C).$$
Fix $t>0$ and $\epsilon>0$.\\
\textbf{First case $x=1$.} Recall that for all $s\in[0,\rho]$ where $\rho=\rho_0$, we have $$K^W_{0,t}(1)=\frac{1}{2}(\delta_{e^{iW^+_t}}+\delta_{e^{-iW^+_t}}).$$
This shows that when $t\leq\rho$, $K_{0,t}(1)$ is supported on $\{e^{iW^+_t},e^{-iW^+_t}\}$ and so $X^1_t=e^{iW^+_t}$ or $e^{-iW^+_t}$. Moreover, by Lemma \ref{rz} (i), if $y\notin\{1,e^{il}\}$, then $X^y_s=ye^{i\varepsilon(y)W_s}$ for all $s\in[0,\tau(y)]$ where $\tau(y)=\tau_0(y)$ .\\
Let $A=\{T^{1,y}>t\}\cap\{d(X^1_t,Y^y_t)>\varepsilon\}$ with $y$ close to $1$ such that $y\neq 1$ and write $$\mathbb Q(A)=\mathbb Q(A\cap\{t\leq\tau(y)\})+\mathbb Q(A\cap\{t>\tau(y)\}).$$
Since $\tau(y)$ tends to $0$ as $y$ goes to $1$, we have $\lim_{y\rightarrow 1}\mathbb Q(A\cap\{t\leq\tau(y)\})=0$. Moreover $$\mathbb Q(A\cap\{t>\tau(y)\})\leq\mathbb Q(B)+\mathbb Q(X^y_{\tau(y)}=e^{il}).$$
where $B=A\cap\{t>\tau(y), X^y_{\tau(y)}=1\}$. Obviously
$$\mathbb Q(B)\leq\mathbb Q(B\cap\{\tau(y)<\rho\})+\mathbb Q(\tau(y)\geq\rho)$$
with $\lim_{y\rightarrow 1}\mathbb Q(\tau(y)\geq\rho)=0$. On $B\cap\{\tau(y)<\rho\}$, we have $X^1_{\tau(y)}=X^y_{\tau(y)}=1$ and thus $T^{1,y}\leq \tau(y)$. As a result $$\mathbb Q(B\cap\{\tau(y)<\rho\})\leq\mathbb Q(t<T^{1,y}\leq\tau(y)).$$
Since the right-hand side converges to $0$ as $y\rightarrow1$, $(C)$ is satisfied for $x=1$.\\
\textbf{Second case $x\neq 1$.} By analogy $(C)$ is satisfied for $x=e^{il}$. Let $x\notin\{1,e^{il}\}$ and $y$ be close to $x$, then $X^x$ and $X^y$ move parallely until one of them reaches $1$ or $e^{il}$ say at time $T$. Since $P^2$ is Feller, the strong Markov property at time $T$ and the established result for $x\in\{1,e^{il}\}$ allows to deduce $(C)$ for $x$.
\end{proof}
\noindent\textbf {Consequences:} By the proof of Theorem 4.2 \cite{MR2060298}, there exists a joint realization $(K^1,K^2)$ on a probability space $(\tilde{\Omega},\tilde{\mathcal A},\tilde{\mathbb P})$ where $K^1$ and $K^2$ are two stochastic flows of kernels satisfying $K^1\overset{law}{=}\delta_{\varphi^c}$, $K^2\overset{law}{=}K$ and such that:
\begin{enumerate}
\item [(i)] $\hat K_{s,t}(x,y)=K^1_{s,t}(x)\otimes K^2_{s,t}(y)$ is a stochastic flow of kernels on $\mathscr C^2$,
\item [(ii)] For all $s\leq t, z\in \mathscr C$, a.s. $K^2_{s,t}(z)=E[K^1_{s,t}(z)|K^2]$.
\end{enumerate}
To simplify notations, we will denote $(K^1,K^2)$ by $(\delta_{\varphi^c},K)$. Recall that (i) and (ii) are also satisfied  by the pair $(\delta_{\varphi},K^{m^+,m^-})$ constructed in Section \ref{ss}. Now (ii) rewrites, for all $s\leq t, z\in \mathscr C$,
\begin{equation}\label{h}
K_{s,t}(z)=E[\delta_{\varphi^{c}_{s,t}(z)}|K]\ \ a.s.
\end{equation}
and using (\ref{k}), we obtain, for all $s\leq t, z\in \mathscr C$,
\begin{equation}\label{label}
K^{W}_{s,t}(z)=E[\delta_{\varphi^{c}_{s,t}(z)}|\sigma(W)]\ a.s.
\end{equation}
with $K^W$ being the Wiener solution.
\subsubsection{The law of $K$.\newline}
\noindent Recall the definitions of $\mathscr C^+$ and $\mathscr C^-$ from (\ref{geor}) and set for all $s\leq t$, $$U^+_{s,t}=K_{s,t}(1,\mathscr C^+)\ \textrm{and} \ U^-_{s,t}=K_{s,t}(e^{il},\mathscr C^-).$$
\begin{prop}\label{wouk}
Recall the definition of $\rho_s$ from (\ref{mom}). Then
\begin{itemize}
\item[(i)] There exist two probability measures $m^+$ and $m^-$ on $[0,1]$ with mean $\frac{1}{2}$ such that for all $s<t$, conditionally to $\{s<t<\rho_s\}$, $U^{\pm}_{s,t}$ is independent of $W$ and has for law $m^{\pm}$. Moreover, for all $s\in\R, z\in\mathscr C$, a.s. $\forall t\in[s,\rho_s]$,
\begin{eqnarray}
K_{s,t}(z)&=&\delta_{ze^{i\epsilon(z)W_{s,t}}} 1_{\{t\leq \tau_s(z)\}}\nonumber\\
&+&\left(K_{s,t}(1)1_{\{ze^{i\epsilon(z)W_{s,\tau_{s}(z)}}=1\}}+K_{s,t}(e^{il})1_{\{ze^{i\epsilon(z)W_{s,\tau_{s}(z)}}=e^{il}\}}\right)1_{\{t>\tau_s(z)\}}\nonumber\
\end{eqnarray}
where
\begin{eqnarray}
K_{s,t}(1)&=&U^+_{s,t}\delta_{\exp(iW^+_{s,t})}+(1-U^{+}_{s,t})\delta_{\exp(-iW^+_{s,t})},\nonumber\\
K_{s,t}(e^{il})&=&U^-_{s,t}\delta_{\exp(i(l+W^-_{s,t}))}+(1-U^{-}_{s,t})\delta_{\exp(i(l-W^-_{s,t}))}.\nonumber\
\end{eqnarray}
\item [(ii)] For all $s<t$, conditionally to $\{\rho_s>t\}$, $U^+_{s,t}, U^-_{s,t}$ and $W$ are independent.
\end{itemize}
\end{prop}
\noindent The proof of (i) essentially follows \cite{MR2235172} and will be deduced after establishing the lemmas \ref{r},\ref{uhb},\ref{hub},\ref{bhu} and \ref{buh} below.\\
For all $-\infty\leq s\leq t\leq +\infty$, define $\mathcal F^{K}_{s,t}=\sigma(K_{u,v}, s\leq u\leq v\leq t)$ and recall the definition of $\mathcal F^{W}_{s,t}$ from (\ref{jij}). When $s=0$, we denote $K_{0,t}, \varphi^c_{0,t}, \mathcal F^{K}_{0,t}, \mathcal F^{W}_{0,t}, U^{\pm}_{0,t}$ simply by $K_{t}, \varphi^c_{t}, \mathcal F^{K}_{t}, \mathcal F^{W}_{t}, U^{\pm}_t$. We will always consider the usual augmentations of these $\sigma$-fields which include all $\mathbb P$-negligible sets and are right-continuous. For each each $z\in\mathscr C$, recall that $t\longmapsto K_{t}(z)$ is continuous from $[0,+\infty[$ into $\mathcal P(\mathscr C)$. Denote by $\mathbb P_z$ the law of $K_{\cdot}(z)$ which is a probability measure on $C(\R_+,\mathcal P(\mathscr C))$, then since $K_{\cdot}(z)$ is a Feller process (see Lemma 2.2 \cite{MR2060298}) the following strong Markov property holds
\begin{lemma}\label{r}
Let $z_1, z_2\in \mathscr C$ and $T$ be a finite $(\mathcal F^{K}_{t})_{t\geq0}$-stopping time. On $\{K_{T}(z_1)=\delta_{z_2}\}$, the law of  $K_{T+\cdot}(z_1)$ knowing $\mathcal F^{K}_{T}$ is given by $\mathbb P_{z_2}$.
\end{lemma}
\noindent Let
$$\rho^+=\inf\{r\geq 0 : W^+_{r}=l\} \quad \textrm{and} \quad L=\sup\{r\in[0,\rho^+] : W^+_{r}=0\}.$$
Thanks to (\ref{label}), on the event $\{0\leq t\leq \rho^+\}$, a.s. $$E[\delta_{\varphi^{c}_{t}(1)}|\sigma(W)]=\frac{1}{2}(e^{iW^+_t}+e^{-iW^+_t}).$$
By the continuity of $\varphi^c_{\cdot}(1)$, this shows that a.s.
\begin{equation}\label{gt}
\forall t\in[0,\rho^+],\quad \varphi^c_{t}(1)\in\{e^{iW^+_t},e^{-iW^+_t}\}.
\end{equation}
Let $h\in C(\mathscr C)$ such that $\forall x\in[-l,l]$, $h(e^{ix})=|x|$.
Using (\ref{h}), the fact that $\sigma(W)\subset\sigma(K)$ and the continuity of $t\longmapsto K_{t}(1)$, we have a.s. $\forall g\in C_0(\R), \forall t\in [0,\rho^+]$,
$$K_{t}(g\circ h)(1)=g(W^+_t).$$
Thus a.s. $\forall t\in[0,\rho^+],\ K_{t}h(1)=W^+_t$ and $\rho^+$ can be expressed as
\begin{equation}\label{63}
\rho^+=\inf\{t\geq 0 : K_{t}h(1)=l\}.
\end{equation}
Define the $\sigma$-fields:
$$\mathcal{F}_{L-}=\sigma(X_{L}, X\ \textrm{is a bounded}\ \  \mathcal{F}^{W}_{0,\cdot}-\textrm{previsible process}),$$
$$\mathcal{F}_{L+}=\sigma(X_{L}, X\ \textrm{is a bounded}\ \  \mathcal{F}^{W}_{0,\cdot}-\textrm{progressive process}).$$
By Lemma 4.11 in \cite{MR2235172}, we have $\mathcal{F}_{L+}=\mathcal{F}_{L-}$. Let $f : \R\longrightarrow\R$ be a bounded continuous function and set $$X_t=E[f(U^+_{t})|\sigma(W)]1_{\{0\leq t\leq \rho^+\}}.$$
\noindent By (\ref{h}), the process $U^+$ is constant on the excursions of $W^+$ out of $0$ before $\rho^+$.
\begin{lemma}\label{uhb}
There exists an $\mathcal F^W$-progressive version of $X$ denoted $Y$ that is constant on the excursions of $W^+$ out of $0$ before $\rho^+$ and satisfies $Y_L=Y_{\rho^+}$ a.s.
\end{lemma}
\begin{proof}
We closely follow Lemma 4.12 \cite{MR2235172} and correct an error at the end of the proof there.
By induction, for all integers $k$ and $n$, define the sequence of stopping times $S_{k,n}$ and $T_{k,n}$ by the relations: $T_{0,n}=0$ and for $k\geq 1$,
\begin{eqnarray}
S_{k,n}&=&\inf\{t\geq T_{k-1,n}: W^+_t=2^{-n}\},\nonumber\\
T_{k,n}&=&\inf\{t\geq S_{k,n}: W^+_t=0\}.\nonumber\
\end{eqnarray}
In the following $U^+_{k,n}$ will denote $U^+_{S_{k,n}}$. For all $t>0$, on $\{t\in[S_{k,n},T_{k,n}[, t\leq \rho^+\}$, we have $U^+_{t}=U^+_{k,n}$ as. Let $X_{k,n}:=E[f(U^+_{k,n})|W]1_{\{S_{k,n}\leq\rho^+\}}$. Since $\sigma(W_{S_{k,n},u+S_{k,n}}, u\geq0)$ is independent of $\mathcal{F}^K_{S_{k,n}}$, we have $X_{k,n}=E[f(U_{k,n})|\mathcal{F}_{S_{k,n}}^{W}]1_{\{S_{k,n}\leq\rho^+\}}$ which is $\mathcal{F}_{S_{k,n}}^{W}$ measurable. Set $I_{n}=\bigcup_{k\geq 1}[S_{k,n},T_{k,n}[$ and define

$$ X_{t}^{n} = \begin{cases}
       X_{k,n} & \text{if} \ t\in[S_{k,n},T_{k,n}[\ (\textrm{for some}\ k)\ \textrm{and}\ t\leq \rho^+,\\
        f(0)& \text{if} \ t\in I_{n}^c\cap[0,\rho^+],\\
         0& \text{if} \ t>\rho^+.\\
       \end{cases}$$
Then $X^{n}$ is $\mathcal{F}^{W}$-progressive. For all $t\geq 0$, set $\tilde X_{t}=\limsup_{n\rightarrow\infty}X_{t}^{n}$, then $\tilde X$ is $\mathcal{F}^{W}$-progressive and for all $t\geq0$, $\tilde X_{t}=X_{t}$ a.s. Indeed, fix $t>0$ and on the event $\{\rho^+>t\}$, choose $k_0$ and $n_0$ such that $t\in[S_{k_0,n_0},T_{k_0,n_0}[$, then $X_{t}^{n_0}=X_{k_0,n_0}$. For all $n\geq n_0$, there exists an integer $l_n$ such that $t\in[S_{l_n,n},T_{l_n,n}[$. Thus $X_{t}^{n}=X_{l_n,n}=X_{k_0,n_0}$ since $S_{k_0,n_0}$ and $S_{l_n,n}$ belong to the same excursion interval of $W^+$ containing also $t$. Now set $Y_0=f(0)$ and $Y_t=\limsup_{n\rightarrow\infty}\tilde X_{t+\frac{1}{n}}$ for all $t>0$. Then $Y$ is a modification of $X$ which is $\mathcal{F}^{W}$-progressive and constant on the excursions of $W^+$ out of $0$ before $\rho^+$. Moreover $Y_L=Y_{\rho^+}$ a.s.
\end{proof}
We take for $X$ this $\mathcal{F}^{W}$-progressive  version. Then $X_{\rho^+}=E[f(U^+_{\rho^+})|\sigma(W)]$ is $\mathcal{F}_{L+}$ measurable.
\begin{lemma}\label{hub}
$E[X_{\rho^+}|\mathcal{F}_{L-}]=E[f(U^+_{\rho^+})]$.
\end{lemma}
\begin{proof}
Let $S$ be an $\mathcal{F}^W$-stopping time and $d_{S}=\inf\{t\geq S: W^+_{t}=0\}$. We have $\{S<L\}=\{d_{S}<\rho^+\}$ (up to some negligible set) and so $\{S<L\}\in\mathcal{F}_{d_{S}}^{W}$. Let $H=d_S\wedge \rho^+$ and $K=\inf\{r\geq0: K_{H+r}h(1)=l\}$, then
$$E[X_{\rho^+}1_{\{d_{S}<\rho^+\}}]=E[f(U^+_{H+K})1_{\{d_{S}<\rho^+, K_{H}(1)=\delta_1\}}].$$
Note that on $\{d_{S}<\rho^+\}$, we have $H+K=\rho^+$ a.s. Applying Lemma \ref{r} at time $H$ and using (\ref{63}), we get
$$E[X_{\rho^+}1_{\{d_{s}<\rho^+\}}]=E[f(U^+_{\rho^+})]E[1_{\{d_{S}<\rho^+, K_{H}(1)=\delta_1\}}]=E[f(U^+_{\rho^+})]\mathbb P(d_{S}<\rho^+).$$
Since the $\sigma$-field $\mathcal{F}_{L-}$ is generated by the events $\{S<L\}$ for all stopping time $S$ (see \cite{MR1780932} page 344), the lemma holds.
\end{proof}
The previous lemma implies that $U^+_{\rho^+}$ is independent of $\sigma(W)$ (Lemma 4.14 \cite{MR2235172}) and the same holds if we replace $\rho^+$ by $\inf\{t\geq 0 : W^+_{t}=a\}$ where $0<a\leq l$. For $n$ such that $2^{-n}<l$, define inductively $T_{0,n}^+=0$ and for $k\geq 1$:
\begin{eqnarray}
S_{k,n}^{+}&=&\inf\{t\geq T_{k-1,n}^{+} : W^+_{t}=2^{-n}\},\nonumber\\
T_{k,n}^{+}&=&\inf\{t\geq S_{k,n}^{+} : W^+_{t}=0\}.\nonumber\
\end{eqnarray}
Set $V_{k,n}^{+}=U_{S_{k,n}^{+}}^{+}$. Then, we have the following
\begin{lemma}\label{bhu}
For all $q\geq 1$, conditionally to $\{S_{q,n}^{+}\leq \rho^+\}$, $V_{1,n}^{+},\cdots, V_{q,n}^{+}, W$ are independent and $V_{1,n}^{+},\cdots, V_{q,n}^{+}$ have the same law (which depends on $n$ but no longer depends on $q$).
\end{lemma}
\begin{proof}
We prove the result by induction on $q$. For $q=1$, this has been justified. Suppose the result holds for $q-1$ and let $(f_j)$ be an approximation of $\epsilon$ as in the proof of Lemma \ref{rz} (ii). For a fixed $t\geq 0$, in $L^2(\mathbb P)$, we have
$$W_{T_{q-1,n}^{+},t+T_{q-1,n}^{+}}=\lim_{j\rightarrow\infty}\left(K_{t+T_{q-1,n}^{+}}f_j(1)-K_{T_{q-1,n}^{+}}f_j(1)-\frac{1}{2}\int_0^tK_{u+T_{q-1,n}^{+}}f_j''(1)du\right).$$

On $\{S_{q,n}^{+}\leq \rho^+\}$, we have $K_{T_{q-1,n}^{+}}(1)=\delta_1$ and therefore, in $L^2(\mathbb P(.|S_{q,n}^{+}\leq \rho^+))$,
\begin{equation}\label{t}
W_{T_{q-1,n}^{+},t+T_{q-1,n}^{+}}=\lim_{j\rightarrow\infty}\left(K_{t+T_{q-1,n}^{+}}f_j(1)-f_j(1)-\frac{1}{2}\int_0^tK_{u+T_{q-1,n}^{+}}f_j''(1)du\right)
\end{equation}
As $2^{-n}<l$, $\{S_{q,n}^{+}\leq \rho^+\}=\{T_{q-1,n}^{+}\leq \rho^+\}$ a.s. Choose a family $\{g_1,\cdots,g_q,g,h\}$ of bounded continuous functions on $\R$. For any $A\in\mathcal A$, we will use the notation $E_A$ to denote the expectation under $\mathbb P(\cdot|A)$. Set $A_{q,n}=\{S_{q,n}^{+}\leq \rho^+\}$, then using (\ref{t}) and Lemma \ref{r} at time $T_{q-1,n}^{+}$, we get
$$E_{A_{q,n}}\left[\prod_{i=1}^{q} g_i(U_{S_{i,n}^{+}}^{+})g(W_{t\wedge T_{q-1,n}^{+}})h(W_{T_{q-1,n}^{+},t+T_{q-1,n}^{+}})\right]$$
$$=E_{A_{q,n}}\left[\prod_{i=1}^{q-1} g_i(U_{S_{i,n}^{+}}^{+})g(W_{t\wedge T_{q-1,n}^{+}})\right] E\left[h(W_t)\right]E\left[g_q(U^+_{S_{1,n}^{+}})\right].$$
Since $A_{q-1,n}\subset A_{q,n}$, we have by the induction hypothesis
$$E_{A_{q,n}}\left[\prod_{i=1}^{q-1} g_i(U_{S_{i,n}^{+}}^{+})g(W_{t\wedge T_{q-1,n}^{+}})\right]=E_{A_{q-1,n}}\left[\prod_{i=1}^{q-1} g_i(U_{S_{i,n}^{+}}^{+})\right]E_{A_{q,n}}\left[g(W_{t\wedge T_{q-1,n}^{+}})\right].$$
In conclusion
$$E_{A_{q,n}}\left[\prod_{i=1}^{q} g_i(U_{S_{i,n}^{+}}^{+})g(W_{t\wedge T_{q-1,n}^{+}})h(W_{T_{q-1,n}^{+},t+T_{q-1,n}^{+}})\right]$$
$$=E_{A_{q-1,n}}\left[\prod_{i=1}^{q-1} g_i(U_{S_{i,n}^{+}}^{+})\right] E_{A_{q,n}}\left[g(W_{t\wedge T_{q-1,n}^{+}})h(W_{T_{q-1,n}^{+},t+T_{q-1,n}^{+}})\right]E\left[g_q(U^+_{S_{1,n}^{+}})\right].$$
The last identity remains satisfied if we replace $g(W_{t\wedge T_{q-1,n}^{+}})h(W_{T_{q-1,n}^{+},t+T_{q-1,n}^{+}})$ by a finite product\\
$\prod_{i=1}^{k}g^i(W_{t_i\wedge T_{q-1,n}^{+}})h^i(W_{T_{q-1,n}^{+},t_i+T_{q-1,n}^{+}})$. As a result, for all bounded continuous $g: C(\R_+,\R)\rightarrow\R$,
$$E_{A_{q,n}}\left[\prod_{i=1}^{q} g_i(U_{S_{i,n}^{+}}^{+})g(W)\right]=E_{A_{q-1,n}}\left[\prod_{i=1}^{q-1} g_i(U_{S_{i,n}^{+}}^{+})\right] E_{A_{q,n}}\left[g(W)\right]E\left[g_q(U^+_{S_{1,n}^{+}})\right].$$
Iterating this relation, yields
$$E_{A_{q,n}}\left[\prod_{i=1}^{q} g_i(U_{S_{i,n}^{+}}^{+})g(W)\right]=\prod_{i=1}^qE\left[g_i(U^+_{S_{1,n}^{+}})\right] E_{A_{q,n}}\left[g(W)\right].$$
In particular, for all $i\in[1,q]$,
$$E_{A_{q,n}}\left[g_i(U_{S_{i,n}^{+}}^{+})\right]=E\left[g_i(U^+_{S_{1,n}^{+}})\right].$$
This completes the proof.
\end{proof}
Let $m_n^{+}$ be the law of $V_{1,n}^{+}$ and $m^+$ be the law of $U_1^+$ under $\mathbb P(.|\rho^+>1)$. Then, we have the
\begin{lemma}\label{buh}
The sequence $(m_n^{+})_{n\geq1}$ converges weakly towards $m^+$. For all $t>0$, under $\mathbb P(\cdot|\rho^+>t)$, $U^+_t$ and $W$ are independent and the law of $U^+_t$ is given by $m^+$.
\end{lemma}
\noindent\begin{proof}
For each bounded continuous function $f : \R\longrightarrow\R,$
$$\begin{array}{ll}
E[f(U^+_{t})|W] 1_{\{0<t<\rho^+\}}&=\displaystyle{\lim_{n\rightarrow \,\infty}}\displaystyle{\sum_{k}}
E\big[1_{\{t\in[S_{k,n}^{+},T_{k,n}^{+}[\}}f(V_{k,n}^{+})|W\big]1_{\{0<t<\rho^+\}}\\
&=\displaystyle{\lim_{n\rightarrow \,\infty}}\displaystyle{\sum_{k}} 1_{\{t\in[S_{k,n}^{+},T_{k,n}^{+}\wedge \rho^+[\}}\left(\int f dm_{n}^+\right)\\
&=\left[1_{\{0<t<\rho^+\}}\displaystyle{\lim_{n\rightarrow \,\infty}}\int f dm_{n}^+\right]\\
\end{array}$$
Consequently $$\displaystyle{\lim_{n\rightarrow \,\infty}}\int f dm_{n}^+=\frac{1}{\mathbb P(\rho^+>t)}E[f(U_t^+) 1_{\{\rho^+>t\}}].$$
The left-hand side no longer depends on $t$, which completes the proof.
\end{proof}
By analogy, we define the measure $m^-$  such that if $\rho^-=\inf\{t\geq0: W^-_t=l\}$, then for all $t>0$, under $\mathbb P(\cdot|\rho^->t)$, $U^-_t$ and $W$ are independent and $U^-_t\overset{law}{=}m^-$. Recall the definition $\rho_0=\inf(\rho^+,\rho^-)$, then for all $t>0$, the law of $U^+_{t}$ (respectively $U^-_{t}$) knowing $\{\rho_0>t\}$ is given by $m^+$ (respectively $m^-$).\\
Now take $s=0$ and fix $z\in\mathscr C$. Similarly to (\ref{gt}), we can deduce from (\ref{label}) that a.s. for all $t\in[0,\rho_0]$,
$$\varphi^c_{t}(z)=ze^{i\epsilon(z)W_t},\ \ \varphi^c_{t}(1)\in\{e^{iW^+_t},e^{-iW^+_t}\}\ \textrm{and}\ \ \varphi^c_{t}(e^{il})\in\{e^{i(l+W^-_t)},e^{i(l-W^-_t)}\}.$$
Note that $\varphi^c$ is constructed such that for all $x,y\in\mathscr C$ as. $\varphi^c_{\cdot}(x)$ and $\varphi^c_{\cdot}(y)$ collide whenever they meet. So a.s. for all $t\in[0,\rho_0]$,
\begin{eqnarray}
\varphi^c_{t}(z)&=&ze^{i\epsilon(z)W_{t}} 1_{\{t\leq \tau_0(z)\}}\nonumber\\
&+&\left(\varphi^c_{t}(1)1_{\{ze^{i\epsilon(z)W_{\tau_{0}(z)}}=1\}}+\varphi^c_{t}(e^{il})1_{\{ze^{i\epsilon(z)W_{\tau_{0}(z)}}=e^{il}\}}\right)1_{\{t>\tau_0(z)\}},\nonumber\
\end{eqnarray}
By (\ref{h}), the second claim of Proposition \ref{wouk} (i) holds.\\
\textbf{Proof of Proposition \ref{wouk} (ii)}
We first prove the following statements: For all $0<s<t$, we have
\begin{itemize}
\item [(a)] Conditionally to $\{s<\rho_0, t<\rho_s\}$, $U^+_{s,t}, U^-_{0,s}, W$ are independent and $U^+_{s,t}$ (resp. $U^-_{0,s}$) has for law $m^+$ (resp. $m^-$).
\item [(b)] Let $$g^{\pm}_t=\sup\{u\in[0,t] : W^{\pm}_u=0\}.$$
Then, conditionally to $\{g^-_t<s<g^+_t, s<\rho_0\}$, $U^+_{0,t}, U^-_{0,t}, W$ are independent and the law of $U^+_{0,t}$
 (resp. $U^-_{0,t}$) is $m^+$ (resp. $m^-$).
\item [(c)] Conditionally to $\{g^-_t<g^+_t, t<\rho_0\}$, $U^+_{0,t}, U^-_{0,t}, W$ are independent.
\item [(d)] Conditionally to $\{t<\rho_0\}$, $U^+_{0,t}, U^-_{0,t}, W$ are independent.
\end{itemize}

\noindent \textbf{(a)} Note that $\{s<\rho_0\}\in \mathcal F^W_{s},\ \{t<\rho_s\}\in \mathcal F^W_{s,+\infty}$ and $\mathcal F^W_{0,+\infty}=\mathcal F^W_{s}\vee\mathcal F^W_{s,+\infty}$ with $\mathcal F^W_{s}\subset \mathcal F^K_{s},\ \mathcal F^W_{0,+\infty}\subset\mathcal F^K_{0,+\infty}$. Now (a) holds from Proposition \ref{wouk} (i) and using the independence of $\mathcal F^K_{s}$ and $\mathcal F^K_{s,+\infty}$. \\
\textbf{(b)} By (a), it suffices to show that on $A=\{g^-_t<s<g^+_t, s<\rho_0\}$ (which is a subset of $\{s<\rho_0, t<\rho_s\}$), a.s. $U^-_{0,t}=U^-_{0,s}$ and $U^+_{0,t}=U^+_{s,t}$. The first equality is clear since $r\longmapsto U^-_{r}$ is constant on the excursions of $W^-$ on $[0,\rho]$ and on $A$, $s$ and $t$ belong to the same excursion of $W^-$. Moreover, on $A$, we have $Z:=\varphi^c_{s}(1)\in\{e^{iW^+_s}, e^{-iW^+_s}\}$ and so $\mathbb P(\cdot|A)$ a.s. $$\tau_s(Z)=\inf\{r\geq s: W_r-\textrm{m}^+_{0,s}=0\}=\inf\{r\geq s: W^+_r=0\}\leq g^+_t.$$
Clearly $\varphi^c_{s,\tau_s(Z)}(Z)=\varphi^c_{s,\tau_s(Z)}(1)=1$ and therefore $\varphi^c_{s,r}(Z)=\varphi^c_{s,r}(1)$ for all $r\geq\tau_s(Z)$ (using the coalescence property of $\varphi^c$ and the independence of increments). On $A$, $\tau_s(Z)\leq g^+_t\leq t$ and by the flow property of $\varphi^c$, a.s.
$$\varphi^c_{t}(1)=\varphi^c_{s,t}(y)=\varphi^c_{s,t}(1).$$
Using (\ref{h}), we get $\mathbb P(\cdot|A)$ a.s. $U^+_{0,t}=U^+_{s,t}$. \\
\textbf{(c)} For all $n\geq0$, let $\mathbb D_n=\{\frac{k}{2^n},\ k\in\N\}$ and $\mathbb D=\cup_{n\in\N}\mathbb D_n$. Define for $0\leq u<v$, $$n(u,v)=\inf\{n\in\N: \mathbb D_n\cap ]u,v[\neq\emptyset\}\quad \textrm{and}\quad f(u,v)=\inf (\mathbb D_{n(u,v)}\cap]u,v[).$$
Then by writing $$\{g^-_t<g^+_t, t<\rho_0\}=\bigcup_{s\in\mathbb D}\{g^-_t<s<g^+_t, t<\rho_0, s=f(g^-_t,g^+_t)\}$$
and using that $f(g^-_t,g^+_t)1_{\{g^-_t<g^+_t\}}$ is $\sigma(W)$-measurable, (c) easily holds from (b).\\
\textbf{(d)} By analogy with (c), conditionally to $\{g^+_t<g^-_t, t<\rho_0\}$, $U^+_{0,t}, U^-_{0,t}, W$ are independent. Now (d) holds after remarking that as. $\{t<\rho_0\}=\{g^-_t<g^+_t, t<\rho_0\}\cup\{g^+_t<g^-_t, t<\rho_0\}$.\\
Finally Proposition \ref{wouk} (ii) holds for $s=0$ and thus for all $s$ using the stationarity of $K$.\\
Now the proof of Proposition \ref{wouk} is completed. \qed
\begin{prop}\label{jhkh}
We have $K\overset{law}{=}K^{m^+,m^-}$.
\end{prop}
\begin{proof}
Like in Section \ref{ss}, extending the probability space, we can construct a flow $K'$ such that $(K',W)$ has the same law as $(K^{m^+,m^-},W)$. By Proposition \ref{wouk}, for all $t>s$, $K_{s,t}\overset{law}{=}K'_{s,t}$ conditionally to $\{\rho_s>t\}$. For $t>0$ and $n\geq 1$, let $t^n_i=\frac{it}{n}, i\in[0,n]$ and define $A_{n,i}=\{t^n_i\leq\rho_{t^n_{i-1}}\}\in \mathcal F^{W}_{t^n_{i-1},t^n_i}$, $A_n=\cap_{i=1}^{n}A_{n,i}$. Then by the independence of increments of $K$ and $K'$,
$$(K_{0,t^n_1},\cdots,K_{t^n_{n-1},t})\overset{law}{=}(K'_{0,t^n_1},\cdots,K'_{t^n_{n-1},t})\ \textrm{on}\ A_n.$$
Recall that $\mathbb P(A_n^{c})\rightarrow0$ as $n\rightarrow\infty$ (see the proof of Proposition \ref{oo}). Letting $n\rightarrow\infty$ and using the flow property for both $K$ and $K'$, we deduce that $K_{0,t}\overset{law}{=}K'_{0,t}$.
\end{proof}
\begin{remark}
Let $\varphi$ be the coalescing flow constructed in Section \ref{TAB}, then $\varphi\overset{law}{=}\varphi^c$. As before this remains to show that  conditionally to $\{\rho_s>t\}$, $\varphi_{s,t}$ is distributed as $\varphi^c_{s,t}$. However the situation is more easy here and we do not need the lemmas \ref{r},\ref{uhb},\ref{hub},\ref{bhu} and \ref{buh}. For example
$$\eta^+_{s,t}=1_{\{\varphi^c_{s,t}(1)\in\mathscr C^+\}}-1_{\{\varphi^c_{s,t}(1)\in\mathscr C^-\}}$$
is independent of $\sigma(|\varphi^c_{s,u}(1)|, s\leq u\leq \rho_s)$ conditionally to $\{\rho_s>t\}$ where $|\cdot|$ is the distance to $1$ since $\varphi^c_{s,\cdot}(1)$ is a Brownian motion on $\mathscr C$. Following Proposition \ref{jhkh}, we check that $\varphi\overset{law}{=}\varphi^c$. In particular $\varphi^c$ solves $(T_{\mathscr C})$.
\end{remark}
\section{Proof of Proposition \ref{telv}}\label{maths}
In this section, we use the same notations as in Section \ref{TAB}. For $r\geq0$, we denote $W^{\pm}_{0,r}$ simply by $W^{\pm}_{r}$. For all $a\in\R, b\geq 0$ define $$T_a=\inf\{r\geq 0: W_{r}=a\}\ \textrm{and}\ \ \gamma^{\pm}_b=\inf\{r\geq 0: W^{\pm}_{r}=b\}.$$ 
\noindent We will further need the following
\begin{lemma}\label{zz}
For all $a>0, b>0$ and $c<0$, we have $\mathbb P(T_a<\gamma^-_b\wedge T_c)>0$.
\end{lemma}
\begin{proof}
Fix $\eta\in]0,\frac{b}{2}\wedge (-c)[$ and let $k\geq 1$ such that $k\eta\geq a$. Now define the sequence of stopping times $(R_i)_{i\geq 0}$ such that $R_0=0$ and for $i\geq 0$, $$R_{i+1}=\inf\{r\geq R_i: |W_{r}-W_{R_i}|=\eta\}.$$
Let $A=\cap_{i=1}^{k}\{W_{R_i}=W_{R_{i-1}}+\eta\}$. Then on $A$, $\sup_{r\leq R_k}W_r=k\eta\geq a$ and for all $i\in[0,k-1], u\in[R_i,R_{i+1}]$, 
$$W^-_{u}=\sup_{r\leq u}W_r-W_u=\sup_{R_i\leq s\leq u}(W_s-W_u)\leq 2\eta<b.$$
Moreover $\inf_{0\leq r\leq R_k} W_r>-\eta\geq c$. Since $A\subset \{T_a<\gamma^-_b\wedge T_c\}$ and $\mathbb P(A)=\frac{1}{2^k}$, this proves the lemma.
\end{proof}
Let $a>0$. Since $\{T_a<\gamma^-_a\wedge T_{-a}\}\subset \{T_a<\gamma^-_a\}$, we deduce that $\mathbb P(T_a<\gamma^-_a)>0$. Obviously $\gamma^+_a\leq T_a$. Since $W\overset{law}{=}-W$, we have $\mathbb P(\gamma^+_a<\gamma^-_a)=\mathbb P(\gamma^-_a<\gamma^+_a)=\frac{1}{2}$. Remark also that 
$$\gamma^+_a\wedge \gamma^-_a=\inf\{r\geq 0: W^+_{r}+W^-_{r}=a\}.$$
This shows that on $\{\gamma^+_a<\gamma^-_a\}$, we have $W^-_{\gamma^+_a}=0$ and similarly on $\{\gamma^-_a<\gamma^+_a\}$, we have $W^+_{\gamma^-_a}=0$.
\subsection{The case $l=\pi$.\label{habilt}\newline}
This is the more easy case. 
\begin{lemma}\label{ho} With probability $1$, for all $z\in\mathscr C$, we have

$$\varphi_{0,\gamma^+_{\pi}}(z)=-1,\qquad K^{m^+,m^-}_{0,\gamma^+_{\pi}}(z)=\delta_{-1}$$
and 
$$\varphi_{0,\gamma^-_{\pi}}(z)=1,\qquad K^{m^+,m^-}_{0,\gamma^-_{\pi}}(z)=\delta_{1}.$$

\end{lemma}
\begin{proof}
This lemma is a consequence of the facts that  $(\varphi_{0,\gamma^+_{\pi}}(1),K^{m^+,m^-}_{0,\gamma^+_{\pi}}(1))=(-1,\delta_{-1})$ and that $(\varphi_{0,\gamma^-_{\pi}}(-1),K^{m^+,m^-}_{0,\gamma^-_{\pi}}(-1))=(1,\delta_{1})$. 
Let us just explain why $\varphi_{0,\gamma^+_{\pi}}(1)=-1$ implies $\varphi_{0,\gamma^+_{\pi}}(z)=-1$ for all $z\in\mathscr C$. 
Fix $z\in\mathscr C$. To simplify, assume $\arg(z)\in [0,\pi]$. 
It holds that $\rho_0=\gamma^+_{\pi}\wedge \gamma^-_{\pi}$ and that $\tau_0(z)\leq \rho_0$. 
Then $\varphi_{0,\rho_0}(z)=\varphi_{0,\rho_0}(-1)$ on the event $\{\arg(z)+W_{\tau_0(z)}=\pi\}$ and $\varphi_{0,\rho_0}(z)=\varphi_{0,\rho_0}(1)$ on the event $\{\arg(z)+W_{\tau_0(z)}=0\}$.
Now, if $\rho_0=\gamma^+_{\pi}<\gamma^-_{\pi}$, then  $\varphi_{0,\gamma^+_{\pi}}(1)=\varphi_{0,\gamma^+_{\pi}}(-1)=-1$ (this thus implies that $\varphi_{0,\gamma^+_{\pi}}(z)=-1$). 
And if $\rho_0=\gamma^-_{\pi}<\gamma^+_{\pi}$, then  $\varphi_{0,\gamma^-_{\pi}}(1)=\varphi_{0,\gamma^-_{\pi}}(-1)=1$ (this thus implies that $\varphi_{0,\gamma^-_{\pi}}(z)=1$). 
To conclude in this case, we use the flow property $\varphi_{0,\gamma^+_{\pi}}(z)=\varphi_{\gamma^-_{\pi},\gamma^+_{\pi}}(\varphi_{0,\gamma^-_{\pi}}(z))=\varphi_{\gamma^-_{\pi},\gamma^+_{\pi}}(1)$. It remains to remark that $\varphi_{\gamma^-_{\pi},\gamma^+_{\pi}}(1)=\varphi_{0,\gamma^+_{\pi}}(1)=-1$.
\end{proof}
\noindent To prove Proposition \ref{telv}, consider the sequences of stopping times given by $S_1=\rho^+_{\pi}$ and for $k\geq 1$,
\begin{eqnarray} 
T_{k}&=&\inf\{u\geq S_k :\; W^-_{S_k,u}=\pi\},\nonumber\\
S_{k+1}&=&\inf\{u\geq T_{k}:\; W^+_{T_{k},u}=\pi\}.\nonumber\
\end{eqnarray}
Then Lemma \ref{ho} implies that $(S_{k})_{k\geq 1}$ (resp. $(T_k)_{k\geq 1}$) satisfies $(1)$ (resp. (2)) of Proposition \ref{telv}.

\subsection{The case $l\neq\pi$.\newline}
The key argument to prove Proposition \ref{telv} in this case is to find some conditions on the path of $W$ under which the image of the whole circle by $\varphi$ at some specific time is reduced to $e^{il}$. 

We fix $\delta>0$ such that $0<l-\delta<l+\delta<\pi$. For any $(\mathcal F^W_{0,\cdot})$-finite stopping time $S$ and $a \in\R$ define 
$$T_{S,a}=\inf\{r\geq S: W_{S,r}=a\}$$
and $$\gamma^{-}_{S,\delta}=\inf\{r\geq S: W_{S,r}^{-}=\delta\}.$$
Let
$$A_S=\{T_{S,2(\pi-l)}<\gamma^{-}_{S,\delta}\}.$$
The event $A_S$ is the event "for all $t\in [S,T_{S,2(\pi-l)}]$ we have $\displaystyle\sup_{s\in[S,t]}W_{S,s}\le W_{S,t} + \delta$".
Setting $T=T_{S,2(\pi-l)}$, this event can be represented by the following figure (Figure 2).
\begin{figure}[h]
\begin{center}
\includegraphics[width=11cm,height=6cm]{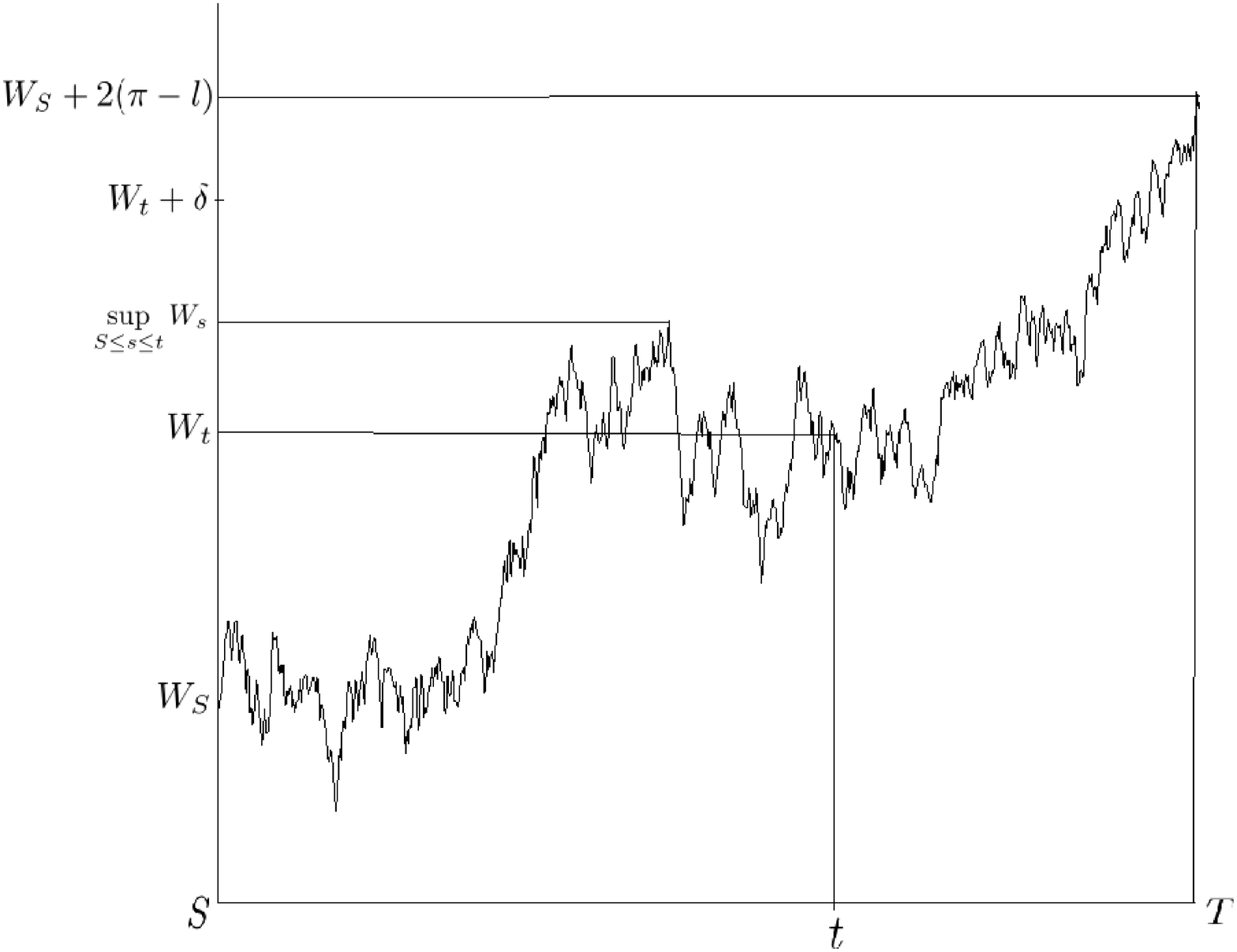}
\caption{The path of W after $S$.}
\end{center}
\end{figure}
On the event $A_S$, $W^-_{S,t}\le \delta$ for all $t\in [S,T]$ and $W_{S,T}=2\pi - 2l$. Thus on this event, we have $\varphi_{S,T}(e^{il})=\varphi_{S,T}(e^{-il})=e^{il}$ and a fortiori $\varphi_{S,T}(z)=e^{il}$ for any intermediate point $z$ such that $\arg(z)\in[l,2\pi-l]$. In other words,
$$A_S\subset\big\{\varphi_{S,\cdot}(e^{-il})\ \textrm{reaches}\ e^{il}\ \textrm{before}\ 1\ \textrm{and before that}\ \varphi_{S,\cdot}(e^{il})\ \textrm{hits}\ e^{i(l+\delta)}\ \textrm{or}\ e^{i(l-\delta)}\big\}.$$ 
Note that $A_S$ is independent of $\mathcal{F}^W_{0,S}$ and that $\mathbb{P}(A_S)>0$ and does not depend on $S$.\\

When $S=\inf\{t\ge 0;\; W^+_{0,t}=l\}$, which is also the first time $t$ when $$\sup_{s\in [0,t]} W_{0,s} - \inf_{s\in [0,t]} W_{0,s} = l.$$
Then at time $S$, we have $\arg(\varphi_{0,S}(z))\in[l,2\pi-l]$ for all $z\in\mathscr C$.  Applying the flow property, we see that on $A_S$, $\varphi_{0,T}(z)=e^{il}$ for all $z\in\mathscr C$. Now the rest of the proof will only require an application of the Borel-Cantelli Lemma. 
We give the details in the following.\\
Define the sequence $(\sigma_k)_{k\geq 0}$ of $(\mathcal F^W_{0,t})_{t\geq0}$-stopping times by $\sigma_0=0$ and for $k\geq 0, \sigma_{k+1}=T_{\rho_{\sigma_k}, 2(\pi-l)}$ (note that $2(\pi-l)=\arg(e^{-il})-\arg(e^{il}))$. Then set, for $k\geq 0$, $$C_{k}=\{W^+_{\sigma_k,\rho_{\sigma_k}}=l\}\cap A_{\rho_{\sigma_k}}.$$
Note that the events $\{W^+_{\sigma_k,\rho_{\sigma_k}}=l\}$ and $A_{\rho_{\sigma_k}}$ are independent. The following proposition describes what happens on $C_k$.
\begin{prop}
With probability $1$, for all $k\geq0$, on $C_k$, we have for all $z\in\mathscr C$,
\begin{enumerate}
 \item [(i)] $\arg(\varphi_{\sigma_k,\rho_{\sigma_k}}(z))\in[l,2\pi-l]$.
\item  [(ii)] If $\arg(z)\in[l,2\pi-l]$, then $\varphi_{\rho_{\sigma_k},\sigma_{k+1}}(z)=e^{il}$.
\item [(iii)] $\varphi_{\sigma_k,\sigma_{k+1}}(z)=e^{il}$.
\item [(iv)] $\varphi_{0,\sigma_{k+1}}(z)=e^{il}$ and $K^{m^+,m^-}_{0,\sigma_{k+1}}(z)=\delta_{e^{il}}$.
\end{enumerate}
\end{prop}
\begin{proof} 
We take $k=0$ (the proof is similar for all $k$). Denote $\rho_0$ simply by $\rho$ and $\rho^n_0$ by $\rho^n$.\\
\textbf{(i)} Fix $z\in\mathscr C$. If $\tau_{0}(z)\leq \rho$, then $\varphi_{0,\rho}(z)\in\{\varphi_{0,\rho}(1),\varphi_{0,\rho}(e^{il})\}$. 
On $C_0$, we have $W^+_{\rho}=l$ and so $W^-_{\rho}=0$ (see the lines after Lemma \ref{zz}). Consequently $\varphi_{0,\rho}(e^{il})=e^{il}$ and $\varphi_{0,\rho}(1)\in\{e^{il},e^{-il}\}$.\\ 
Suppose $\rho<\tau_{0}(z)$, then necessarily $\arg(z)\in]l,2\pi[$ and using that $W_{\rho}=l+\displaystyle\inf_{0\leq u\leq \rho}W_{u}$, we have
$$\varphi_{0,\rho}(z)=\exp(i(\arg(z)-W_{\rho}))=\exp(i(\arg(z)-l-\inf_{0\leq u\leq \rho}W_{u})).$$
Since $\rho<\tau_{0}(z)$, we have $\arg(z)-\displaystyle\inf_{0\leq u\leq \rho}W_{u}<2\pi$ and therefore $\arg(\varphi_{0,\rho}(z))< 2\pi-l$. It is also clear that $\arg(\varphi_{0,\rho}(z))\geq l$ which proves the first statement.\\
\textbf{(ii)}  Let $z\in\mathscr C$ with $\arg(z)\in[l,2\pi-l]$. Then $\varphi_{\rho,\cdot}(e^{-il})$ arrives to $e^{il}$ before $1$ and this happens at time $\sigma_1$. Thus $\varphi_{\rho,\cdot}(z)$ reaches $e^{il}$ before $\sigma_1$. Let $n$ be the greatest integer such that $\rho^n_{\rho}(=\rho^{n+1})\leq \sigma_1$. Then $\varphi_{\rho,\sigma_1}(z)=\varphi_{\rho^{n+1},\sigma_1}(Z)$ where $Z=\varphi_{\rho,\rho^{n+1}}(z)$. Clearly $\tau_{\rho^{n+1}}(Z)=\tau_{\rho}(z)\leq\sigma_1$. Therefore $\varphi_{\rho,\sigma_1}(z)=\varphi_{\rho^{n+1},\sigma_1}(e^{il})$. But $-W_{\rho,u}+2(\pi-l)\geq W^-_{\rho,u}$ for all $u\in[\rho,\sigma_1]$ and so $W^-_{\rho,\sigma_1}=0$. As $\rho^{n+1}\geq\rho$, we get $W^-_{\rho^{n+1},\sigma_1}=0$. That is $\varphi_{\rho,\sigma_1}(z)=e^{il}$.\\
\textbf{(iii)} and \textbf{(iv)} are immediate from the flow property (Corollary \ref{g}) and (i), (ii). The result for $K^{m^+,m^-}$ can be proved by following the same steps with minor modifications.
\end{proof}
\noindent Since for all $k\geq 0$, $\sigma_k$ is an $(\mathcal F^W_{0,t})_{t\geq0}$-stopping time, the sequence $(C_k)_{k\geq0}$ is independent. We also have $\mathbb P(C_k)=\mathbb P(C_0)=\mathbb P(A_0)\times \mathbb P(W^+_{\rho}=l)$ for all $k\geq0$. 
By Lemma \ref{zz}, $\sum_{k\geq0}\mathbb P(C_k)=\infty$ and the Borel-Cantelli lemma yields $\mathbb P(\overline{\lim}C_k)=1$.
\noindent We deduce that with probability 1, $$\varphi_{0,\sigma_k}(\mathscr C)=e^{il}\ \textrm{and}\ \ K^{m^+,m^-}_{0,\sigma_k}(\mathscr C)=\delta_{e^{il}} \ \textrm{for infinitely many}\ k.$$
To deduce Proposition \ref{telv}, we only need to extract from $(\sigma_k)_k$ a subsequence $(\sigma'_k)$ with the preceding property satisfied for all $k$ and not just for infinitely many $k$. This is the subject of the following
\begin{lemma}
Let $(k_n)_{n\geq 0}$ be the sequence of random integers defined by $k_0(\omega)=0$ and for $n\geq 0$, $$k_{n+1}(\omega)=\inf\{k>k_n(\omega) : \omega\in C_k\}.$$
Set $\sigma'_n=\sigma_{k_n}, n\geq 1$. Then $(\sigma'_n)_{n\geq 1}$ is a sequence of $(\mathcal F^W_{0,t})_{t\geq0}$-stopping times such that a.s. $\lim_{n\rightarrow\infty} \sigma'_n=+\infty$, $\varphi_{0,\sigma'_n}(\mathscr C)=e^{il}$ and $K^{m^+,m^-}_{0,\sigma'_n}(\mathscr C)=\delta_{e^{il}}$ for all $n\geq 1$.
\end{lemma}
\begin{proof}
Remark that $C_k\in \mathcal F^W_{\sigma_{k+1}}$ for all $k\geq 0$. For all $n\geq 1$ and $t\geq 0$, we have
$$\{\sigma_{k_n}\leq t\}=\displaystyle\cup _{k\geq 1}\{\sigma_k\leq t, k_n=k\}.$$
It remains to prove that $\{k_n=k\}\in \mathcal F^W_{\sigma_{k+1}}$. We will prove this by induction on $n$. For $n=1$, this is clear since $\{k_1=1\}=C_1$ and for $k\geq 2$, 
$$\{k_1=k\}=C_1^c\cap\cdots\cap C_{k-1}^c\cap C_k.$$
Suppose the result holds for $n$. Then for all $k\geq 2$, 
$$\{k_{n+1}=k\}=\displaystyle\cup _{1\leq i\leq k-1}\left(\{k_{n}=i\}\cap C_{i+1}^c\cap \cdots C_{k-1}^c\cap C_k\right)$$
and the desired result holds for $n+1$ using the induction hypothesis.
\end{proof}

\noindent We have proved Part $(1)$ of Proposition \ref{telv} (for both $\varphi$ and $K^{m^+,m^-})$. Part $(2)$ can be deduced by analogy.
\section{The support of $K^{m^+,m^-}$ (Proof of Proposition \ref{nahi})}\label{yiu}
In this section $\rho^k_0$ and $K^{m^+,m^-}$ will be denoted simply by $\rho^k$ and $K$.
\subsection{The case $l=\pi$.\newline}
When $m^+$ and $m^-$ are both different from $\frac{1}{2}(\delta_0+\delta_1)$, a precise description of $\textrm{supp} (K_{0,t}(1))$ can be given as follows. Recall the definitions of the sequences $(S_k)_{k\geq 1}$ and $(T_k)_{k\geq 1}$ from Section \ref{habilt} and set $T_0=0$. Then for all $k\in\N, t\in[T_k,S_{k+1}]$,
$$\textrm{supp} (K_{0,t}(1))=\{e^{iW^+_{T_k,t}},e^{-iW^+_{T_k,t}}\}$$
and for all $k\geq 1, t\in[S_k,T_{k}]$,
$$\textrm{supp} (K_{0,t}(1))=\{e^{i(\pi+W^-_{S_k,t})},e^{i(\pi-W^-_{S_k,t})}\}.$$
In fact, for all $s\leq t$,
$$\textrm{supp} (K_{s,t}(1))=\{e^{iX_{s,t}},e^{-iX_{s,t}}\},$$
with $X_{s,t}$ being the unique reflecting Brownian motion on $[0,\pi]$ (see \cite{MR1771660}) solution of 
$$X_{s,t}=W_{s,t}+L^0_{s,t}-L^{\pi}_{s,t},\ \ t\geq s,$$
and
$${L}_{s,t}^{x}=\lim_{\varepsilon \rightarrow 0^+} \frac{1}{2\varepsilon}\int_{s}^{t} 1_{\{|X_{s,u}-x|\leq \varepsilon\}}du,\ \ x=0,\pi.$$ 
If $m^+=m^-=\delta_{\frac{1}{2}}$, then $K$ is a Wiener flow such that $K_{s,t}(1)=\frac{1}{2}(\delta_{e^{iX_{s,t}}}+\delta_{e^{-iX_{s,t}}})$ for all $s\leq t$. 
\subsection{The case $l\neq\pi$.\newline}
From the definition of $K$, $K_{\rho^k,t}(z)$ is carried by at most two points for all $k\geq0$, $t\in[\rho^k,\rho^{k+1}]$ and $ z\in\mathscr C$. Using the flow property and the fact that $\lim_{k\rightarrow\infty}\rho^k=\infty$ a.s., it is therefore clear that a.s.  $$\forall t\geq0,\ z\in\mathscr C,\ \textrm{Card supp}\ K_{0,t}(z)<\infty.$$
We assume in this section that $m^+$ and $m^-$ are both distinct from $\frac{1}{2}(\delta_{0}+\delta_1)$ (for the other case, see Remark \ref{rema} below).\\
Fix a decreasing positive sequence $(\alpha_k)_{k\geq1}$ such that $\alpha_1<\inf(l,2(\pi-l))$. Now define $A_1=\{\ W^+_{0,\rho^1}=l\}$
and for $k\geq 1$,
\begin{eqnarray}
A_{2k}&=&\{W^-_{\rho^{2k-1},\rho^{2k}}=l,\ \alpha_{2k}<\displaystyle\sup_{\rho^{2k-1}\leq u\leq \rho^{2k}}W_{\rho^{2k-1},u}<\alpha_{2k-1}\}\nonumber\\
&=& \{W^-_{\rho^{2k-1},\ \rho^{2k}}=l, -l+\alpha_{2k}<W_{\rho^{2k-1},\rho^{2k}}<-l+\alpha_{2k-1}\}, \nonumber\\
A_{2k+1}&=&\{W^+_{\rho^{2k},\rho^{2k+1}}=l,\ -\alpha_{2k}<\displaystyle\inf_{\rho^{2k}\leq u\leq \rho^{2k+1}} W_{\rho^{2k},u}<-\alpha_{2k+1}\}\nonumber\\
&=&\{W^+_{\rho^{2k},\rho^{2k+1}}=l,\ l-\alpha_{2k}<W_{\rho^{2k},\rho^{2k+1}}<l-\alpha_{2k+1}\}.\nonumber\
\end{eqnarray}
We are going to prove the following
\begin{prop}\label{NES}
Let $C_0=\Omega$ and $C_n=\cap_{i=1}^{n}A_i$ for all $n\geq 1$. Then for all $n\geq 0$,
\begin{enumerate}
 \item [(i)] $\mathbb P(C_n)>0$,
\item[(ii)] $\textrm{Card supp}\left(K_{0,\rho^n}(1)\right)=n+1$ a.s. on $C_n$.
\end{enumerate}
Moreover a.s. for all $k\geq 0,$
\begin{enumerate}
 \item [(ii1)] On $C_{2k}$, $$\textrm{supp}\left(K_{0,\rho^{2k}}(1)\right)=\{P^{2k}_i, 1\leq i\leq 2k+1\},$$ with $\arg(P^{2k}_i)<\arg(P^{2k}_{i+1})$ for all $i\in[1,2k]$,
$$P^{2k}_1=1,\ P^{2k}_2=e^{2il}\ \textrm{and}\ \ P^{2k}_{2k+1}=e^{i(-l-W_{\rho^{2k-1},\rho^{2k}})}.$$
(Note that $\arg(P^{2k}_{2k+1})<2\pi-\alpha_{2k}$.)
\item[(ii2)] On $C_{2k+1}$, we have 
$$\textrm{supp}\left(K_{0,\rho^{2k+1}}(1)\right)=\{P^{2k+1}_i,\ 1\leq i\leq 2k+2\},$$
with $\arg(P^{2k+1}_i)<\arg(P^{2k+1}_{i+1})$ for all $i\in[1,2k+1]$, $$P^{2k+1}_1=e^{il},\ P^{2k+1}_2=e^{i(2l-W_{\rho^{2k},\rho^{2k+1}})}\ \ \textrm{and}\ \ P^{2k+1}_{2k+2}=e^{-il}.$$
(Note that $\arg(P^{2k+1}_2)>l+\alpha_{2k+1}$.)
\end{enumerate}
\end{prop}
\noindent To prove this proposition, let us first establish the following
\begin{lemma}\label{kar}
Fix $0<\alpha<\beta<l$ and define $$E=\{W^-_{\rho}=l,\; \alpha<\displaystyle\sup_{0\leq u\leq \rho}W_{u}<\beta\}$$ where $\rho=\inf\{r\geq 0: \sup(W^{+}_{r},W^{-}_{r})=l\}$. Then $\mathbb P(E)>0.$
\end{lemma}
\begin{proof}
Recall the definition of $T_a$ from the begining of Section \ref{maths}. Consider the event
$$F=\{T_{\alpha}<T_{\beta-l}<T_{\beta}\}\cap{\{\textrm{after}\ T_{\beta-l}, W\ \textrm{reaches}\ \alpha-l\ \textrm{before}\ \beta-l+\alpha\}}.$$
Using the Markov property at time $T_{\beta-l}$, we have $\mathbb P(F)>0$. Note that $\rho$ can be expressed as 
$$\rho=\inf\{t\geq 0: \sup_{0\leq u\leq t}W_u-\inf_{0\leq u\leq t}W_u=l\}.$$

\noindent On $F$, we have $T_{\beta-l}<\rho\leq T_{\alpha-l}$ and so $\alpha<\displaystyle\sup_{0\leq u\leq \rho}W_{u}<\beta$. Moreover, on $F$

$$W^+_{\rho}=W_{\rho}-\inf_{0\leq u\leq t}W_u<\beta-l+\alpha-(\alpha-l)<l.$$
In other words $W^-_{\rho}=l$ which proves the inclusion $F\subset E$ and allows to deduce the lemma.
\end{proof}

\noindent\textbf{Proof of Proposition \ref{NES}} 
\textbf{(i)} The sequence $(A_i)_{i\geq 1}$ is independent and therefore we only need to check that $\mathbb P(A_n)>0$ for all $n\geq 1$. But this is immediate from Lemma \ref{kar} for $n$ even. By replacing $W$ with $-W$, it is also immediate for $n$ odd.\\ 
\textbf{(ii)} We denote the properties (ii1) and (ii2) respectively by $\mathcal P_{2k}$ and $\mathcal P_{2k+1}$. Let prove all the $(\mathcal P_{i})_{i\geq 0}$ by induction. First $\mathcal P_{0}$ and $\mathcal P_{1}$ are clearly satisfied since $K_{0,0}(1)=\delta_1$ and $\textrm{supp}\ K_{0,\rho^1}(1)=\{e^{il}, e^{-il}\}$ on $C_1$. Suppose that all the $\mathcal P_{i}$ hold for all $0\leq i\leq 2k-1$ where $k\geq 1$. On $C_{2k}$, $K_{\rho^{2k-1},t}(e^{-il})\neq \delta_1$ for all $t\in[\rho^{2k-1},\rho^{2k}]$ since for all $t\in ]\rho^{2k-1},\rho^{2k}]$, we have $$-W_{\rho^{2k-1},t}<W^-_{\rho^{2k-1},t}\leq l.$$
Moreover, on $C_{2k}$, we have
$$\inf _{\rho^{2k-1}\leq t\leq \rho^{2k}}\big(2l-W_{\rho^{2k-2},\rho^{2k-1}}-W_{\rho^{2k-1},t}\big)=l-W_{\rho^{2k-2},\rho^{2k-1}}-W_{\rho^{2k-1},\rho^{2k}}>l.$$
Thus for all $t\in[\rho^{2k-1},\rho^{2k}]$, we have $$K_{\rho^{2k-1},t}(P^{2k-1}_2)=e^{i(2l-W_{\rho^{2k-2},\rho^{2k-1}}-W_{\rho^{2k-1},t})}\neq e^{il}$$  
so that $\mathcal P_{2k}$ holds. Similarly, on $C_{2k+1}$, $K_{\rho^{2k},\cdot}(e^{2il})$ cannot reach $\delta_{e^{il}}$ before $\rho^{2k+1}$ since for all $t\in]\rho^{2k},\rho^{2k+1}]$, $$W_{\rho^{2k},t}<W^+_{\rho^{2k},t}\leq l.$$
Moreover, on $C_{2k+1}$,
$$\sup_{\rho^{2k}\leq u\leq \rho^{2k+1}}\big(2\pi-l-W_{\rho^{2k-1},\rho^{2k}}-W_{\rho^{2k},u}\big)=2\pi-(W_{\rho^{2k-1},\rho^{2k}}+W_{\rho^{2k},\rho^{2k+1}})<2\pi.$$
Thus, on $C_{2k+1}$, $K_{\rho^{2k},t}(P^{2k}_{2k+1})\neq \delta_1$ for all $t\in[\rho^{2k},\rho^{2k+1}]$ and $\mathcal P_{2k+1}$ easily holds.
\begin{remark}\label{rema}
When $m^+\neq m^-, m^-= \frac{1}{2}(\delta_{0}+\delta_1)$, by considering
$$E_{2i-1}=A_{2i-1}\quad \textrm{and}\quad E_{2i}=A_{2i}\cap\{K_{\rho^{2i-1},\rho^{2i}}(e^{il})=\delta_1\}\quad \textrm{for }\ i\geq1,$$
and then $F_n=\cap_{1\leq i\leq n}E_i$, we similarly show that $\text{supp}(K_{0,t}(1))$ may be sufficiently large with positive probability.
\end{remark}

\section*{Acknowledgement}
The first author is grateful to Yves Le Jan who suggested to him this problem as a part of his Ph.D.

\bibliographystyle{plain}
\bibliography{Bil}

\begin{thebibliography}{1}

\bibitem{MR1771660}
R.~F. Bass and E.~P. Hsu.
\newblock Pathwise uniqueness for reflecting brownian motion in euclidian
  domains.
\newblock {\em Probab. Theory and Rel. Fields 117, pp. 183-200.},
  117(3):183--200, 2000.
\newblock \href{http://www.ams.org/mathscinet-getitem?mr=1771660}{MR1771660}.

\bibitem{MR2835247}
Hatem Hajri.
\newblock Stochastic flows related to {W}alsh {B}rownian motion.
\newblock {\em Electronic journal of probability 16, 1563-1599}, 2011.
\newblock \href{http://www.ams.org/mathscinet-getitem?mr=2835247}{MR2835247}.

\bibitem{MR1725357}
Ioannis Karatzas and Steven~E. Shreve.
\newblock {\em Brownian motion and stochastic calculus}, volume 113 of {\em
  Graduate Texts in Mathematics}.
\newblock Springer-Verlag, New York, second edition, 1991.
\newblock \href{http://www.ams.org/mathscinet-getitem?mr=1725357}{MR1725357}.

\bibitem{MR1905858}
Yves Le~Jan and Olivier Raimond.
\newblock Integration of {B}rownian vector fields.
\newblock {\em Ann. Probab.}, 30(2):826--873, 2002.
\newblock \href{http://www.ams.org/mathscinet-getitem?mr=1905858}{MR1905858}.

\bibitem{MR2060298}
Yves Le~Jan and Olivier Raimond.
\newblock Flows, coalescence and noise.
\newblock {\em Ann. Probab.}, 32(2):1247--1315, 2004.
\newblock \href{http://www.ams.org/mathscinet-getitem?mr=2060298}{MR2060298}.

\bibitem{MR2235172}
Yves Le~Jan and Olivier Raimond.
\newblock Flows associated to {T}anaka's {SDE}.
\newblock {\em ALEA Lat. Am. J. Probab. Math. Stat.}, 1:21--34, 2006.
\newblock \href{http://www.ams.org/mathscinet-getitem?mr=2235172}{MR2235172}.

\bibitem{MR1780932}
L.~C.~G. Rogers and David Williams.
\newblock {\em Diffusions, {M}arkov processes, and martingales. {V}ol. 2}.
\newblock Cambridge Mathematical Library. Cambridge University Press,
  Cambridge, 2000.
\newblock \href{http://www.ams.org/mathscinet-getitem?mr=1780932}{MR1780932}.

\bibitem{MR1816931}
S.~Watanabe.
\newblock The stochastic flow and the noise associated to {T}anaka's stochastic
  differential equation.
\newblock {\em Ukra\"\i n. Mat. Zh.}, 52(9):1176--1193, 2000.
\newblock \href{http://www.ams.org/mathscinet-getitem?mr=1816931}{MR1816931}.

\end{thebibliography}

\end{document}